\documentclass[12pt]{amsart}
\usepackage{euler, amsfonts, amssymb, latexsym, epsfig,epic}

\usepackage{amscd,amssymb}
\usepackage{verbatim}
\usepackage{pstricks}
\usepackage{pst-node}
\input{diagrams}

\usepackage{bibentry}
\usepackage{amsmath}
\usepackage{amsthm}
\usepackage{amssymb}
\usepackage{amsfonts}
\usepackage{amsxtra}
\usepackage{amscd}
\usepackage{epsfig}
\usepackage{verbatim}

\usepackage{latexsym,amstext,epsfig}
\usepackage[all, knot]{xy}
\xyoption{arc}

\newsymbol\pp 1275
\newcommand{\Hom}{\operatorname{Hom}}
\newcommand{\End}{\operatorname{End}}
\newcommand{\Ext}{\operatorname{Ext}}
\newcommand{\ext}{\operatorname{ext}}
\newcommand{\rep}{\operatorname{rep}}
\newcommand{\Proj}{\operatorname{Proj}}
\newcommand{\SI}{\operatorname{SI}}
\newcommand{\SL}{\operatorname{SL}}
\newcommand{\GL}{\operatorname{GL}}
\newcommand{\PGL}{\operatorname{PGL}}

\newcommand{\ZZ}{\mathbb Z}

\newcommand{\RR}{\mathbb R}

\newcommand{\QQ}{\mathbb Q}
\newcommand{\PP}{\mathbb P}
\newcommand{\coker}{\operatorname{coker}}
\newcommand{\Ima}{\operatorname{Im}}
\newcommand{\Id}{\operatorname{Id}}

\newcommand{\Mat}{\operatorname{Mat}}

\newcommand{\rel}{\operatorname{relint}}

\newcommand{\filt}{\operatorname{filt}}
\newcommand{\iso}{\operatorname{iso}}

\newcommand{\ddim}{\operatorname{\mathbf{dim}}}
\newcommand{\dd}{\operatorname{\mathbf{d}}}
\newcommand{\ee}{\operatorname{\mathbf{e}}}
\newcommand{\hh}{\operatorname{\mathbf{h}}}

\newcommand{\M}{\operatorname{\mathcal{M}}}
\newcommand{\Eff}{\operatorname{Eff}}
\newcommand{\module}{\operatorname{mod}}

\newcommand{\trdeg}{\operatorname{tr.deg}}
\newcommand{\rk}{\operatorname{rk}}
\newcommand{\ddeg}{\operatorname{deg}}
\newcommand{\qq}{\operatorname{Quot}}

\newtheorem{theorem}{Theorem}[section]
\newtheorem{proposition}[theorem]{Proposition}
\newtheorem{corollary}[theorem]{Corollary}
\newtheorem{lemma}[theorem]{Lemma}

\theoremstyle{definition}
\newtheorem{definition}[theorem]{Definition}
\newtheorem{remark}[theorem]{Remark}
\newcount\cols
{\catcode`,=\active\catcode`|=\active
\gdef\Young(#1){\hbox{$\vcenter
{\mathcode`,="8000\mathcode`|="8000
\def,{\global\advance\cols by 1 &}%
\def|{\cr
      \multispan{\the\cols}\hrulefill\cr
       &\global\cols=2 }%
  \offinterlineskip\everycr{}\tabskip=0pt
  \dimen0=\ht\strutbox \advance\dimen0 by \dp\strutbox
    \halign
    {\vrule height \ht\strutbox depth \dp\strutbox##
      &&\hbox to \dimen0{\hss$##$\hss}\vrule\cr
     \noalign{\hrule}&\global\cols=2 #1\crcr
     \multispan{\the\cols}\hrulefill\cr%
   }
}$}} }

\begin{document}

\pagestyle{plain}

\mbox{}
\title{Geometric characterizations of the representation type of hereditary algebras and of canonical algebras}
\author{Calin Chindris}

\address{University of Missouri, Department of Mathematics, Columbia, MO 65211, USA}
\email[Calin Chindris]{chindrisc@missouri.edu}

%\markboth{b}{b}
\date{September 26, 2010; Revised: \today}

\bibliographystyle{plain}
\subjclass[2000]{Primary 16G20; Secondary 16G10, 16G60, 16R30}
\keywords{Canonical algebras, exceptional sequences, moduli spaces, rational invariants, representation type, semi-invariants, tame algebras}
\maketitle

\begin{abstract} We show that a finite connected quiver $Q$ with no oriented cycles is tame if and only if for each dimension vector $\dd$ and each integral weight $\theta$ of $Q$, the moduli space $\M(Q,\dd)^{ss}_{\theta}$ of $\theta$-semi-stable $\dd$-dimensional representations of $Q$ is just a projective space. In order to prove this, we show that the tame quivers are precisely those whose weight spaces of semi-invariants satisfy a certain log-concavity property. Furthermore, we characterize the tame quivers as being those quivers $Q$ with the property that for each Schur root $\dd$ of $Q$, the field of rational invariants $k(\rep(Q,\dd))^{\GL(\dd)}$ is isomorphic to $k$ or $k(t)$. Next, we extend this latter description to canonical algebras. More precisely, we show that a canonical algebra $\Lambda$ is tame if and only if for each generic root $\dd$ of $\Lambda$ and each indecomposable irreducible component $C$ of $\rep(\Lambda,\dd)$, the field of rational invariants $k(C)^{\GL(\dd)}$ is isomorphic to $k$ or $k(t)$. Along the way, we establish a general reduction technique for studying fields of rational invariants on Schur irreducible components of representation varieties. 
\end{abstract}

\vspace{10pt}

\begin{list}{\arabic{enumi}.}
           {\leftmargin=5ex \rightmargin=2ex \usecounter{enumi}
	    \itemsep=-.35mm \topsep=-1.5mm}
\item
{\sf Introduction}
	\hfill\pageref{intro-sec}
\item
{\sf Quiver invariant theory}
	\hfill\pageref{QIT-sec}
\item
{\sf Proof of Theorem \ref{reptype-quivers-thm}}	
\hfill\pageref{proof-thm-sec}

\item
{\sf Quivers with relations}
	\hfill\pageref{quiverel-sec}
\item
{\sf Exceptional sequences and rational invariants}
         \hfill\pageref{excep-ratio-inv-sec}

\item
{\sf Canonical algebras}
          \hfill\pageref{canonical-sec}

\item
{\sf References}
	\hfill\pageref{biblio-sec}
\end{list} \vspace{1.5mm}

\section{Introduction}\label{intro-sec}

Throughout this paper, we work over an algebraically closed field $k$ of characteristic zero. One of the fundamental problems in the representation theory of algebras is that of classifying the indecomposable representations. The representation type of a finite-dimensional algebra reflects the complexity of its indecomposable representations. An algebra is of \emph{tame representation type} if, for each dimension $d$, all but a finite number of $d$-dimensional indecomposable representations belong to a finite number of $1$-parameter families. Within the class of tame algebras, we distinguish the subclass of algebras of \emph{finite representation type}; these are the algebras with only finitely many indecomposable representations up to isomorphism. An algebra is of \emph{wild representation type} if its representation theory is at least as complicated as that of a free algebra in two variables. The remarkable Tame-Wild Dichotomy Theorem of Y. Drozd \cite{Dro} says that every finite-dimensional algebra is of tame representation type or wild representation type and these types are mutually exclusive. Since the representation theory of a free algebra in two variables is known to be undecidable, one can hope to meaningfully classify the indecomposable representations only for tame algebras. For more precise definitions, see \cite[Section 4.4]{Benson} and the reference therein.

The tame quivers are well understood. P. Gabriel's famous result \cite{Ga} identifies the connected quivers of finite representation type as being those whose underlying graphs are the Dynkin diagrams of types $\mathbb{A}$, $\mathbb{D}$, or $\mathbb{E}$. Later on, L. A. Nazarova \cite{Naz}, and P. Donovan and M. R. Freislich \cite{DF} found the representation-infinite tame connected quivers. Their underlying graphs are the Euclidean diagrams of types $\widetilde{\mathbb{A}}$, $\widetilde{\mathbb{D}}$, or $\widetilde{\mathbb{E}}$.

In this paper, we seek for an interpretation of the representation type of an algebra in terms of its (birational) invariant theory. A first result in this direction was obtained by A. Skowro{\'n}ski and J. Weyman in \cite{SW1} where they showed that a finite-dimensional algebra of global dimension one is tame if and only if all of its algebras of semi-invariants are complete intersections. Unfortunately, this result does not extend to algebras of higher global dimension. In fact, W. Kra{\'s}kiewicz \cite{Kra} found examples of algebras of global dimension two for which \cite[Theorem 1]{SW1} does not hold. As it was suggested by Weyman \cite{Jerzy}, in order to detect the tameness of an algebra, one should impose geometric conditions on the various moduli spaces of semi-stable representations rather than on the entire algebras of semi-invariants.

We begin with the following characterization of the tameness of finite-dimensional path algebras.

\begin{theorem} \label{reptype-quivers-thm} Let $Q$ be a finite, connected quiver without oriented cycles. The following conditions are equivalent:
\begin{enumerate}
\renewcommand{\theenumi}{\arabic{enumi}}
\item the path algebra $kQ$ is tame;
\item for each dimension vector $\dd$ and each integral weight $\theta$ of $Q$ such that $\dd$ is $\theta$-semi-stable, $\M(Q,\dd)^{ss}_{\theta}$ is a projective space;
\item for each dimension vector $\dd$ and each integral weight $\theta$ of $Q$, the sequence $\{\dim_k \SI(Q,\dd)_{N\theta}\}_{N \geq 0}$ is log-concave, i.e., $$\dim_k\SI(Q,\dd)_{(N+1)\theta}\cdot \dim_k\SI(Q,\dd)_{(N-1)\theta} \leq (\dim_k\SI(Q,\dd)_{N\theta})^2, \forall N \geq 1;$$
\item for each Schur root $\dd$ of $Q$, the field of rational invariants $k(\rep(Q,\dd))^{\GL(\dd)}$ is isomorphic to $k$ or $k(t)$.
\end{enumerate}
\end{theorem}

We point out that the implication $(1) \Longrightarrow (2)$ was proved by M. Domokos and H. Lenzing by first studying moduli spaces of regular representations for concealed-canonical algebras (see \cite{DL2}). The other implication $(2) \Longrightarrow (1)$ has been  recently proved by Domokos in \cite{Domo2} using the local quiver technique of J. Adriaenssens and L. Le Bruyn (see \cite{ALeB}). Our proof of $(1) \Longleftrightarrow (2)$ is different from the one in \cite{Domo2, DL2}. More specifically, we work entirely within the category of representations of the quiver in question and use in a fundamental way: $(i)$ the study of the log-concavity property for weight spaces of semi-invariants which, in turn, was motivated by A. Okounkov's log-concavity ex-conjecture (see \cite{CDW}); $(ii)$ the H. Derksen and Weyman's notion of $\theta$-stable decomposition for dimension vectors (see \cite{DW2}). Regarding the implication $(1) \Longrightarrow (4)$, we want to point out that a proof can also be obtained from the work of C. Ringel  \cite{R4} on rational invariants for tame quivers, or from the work of A. Schofield \cite{S3} on the birational classification of moduli spaces of representations for quivers. Our proof of $(1) \Longrightarrow (4)$ follows from the general reduction result described in Theorem \ref{rational-inv-quiverel-thm} below (see also Corollary \ref{rational-inv-Euclidean-coro}).

Our next goal in this paper is to extend the equivalence $(1)\Longleftrightarrow (4)$ of Theorem \ref{reptype-quivers-thm} to other classes of algebras. A fundamental role in achieving this goal is played by the following reduction technique. Let $\Lambda$ be the bound quiver algebra of a bound quiver $(Q,R)$ and let $\mathcal{E}=(E_1, \dots, E_t)$ be an orthogonal exceptional sequence of finite-dimensional representations of $\Lambda$. Using the $A_{\infty}$-formalism, one can show that $\mathcal{E}$ gives rise to a triangular algebra $\Lambda_{\mathcal{E}}$ and an equivalence $F_{\mathcal{E}}$ of categories from $\rep(\Lambda_{\mathcal{E}})$ to the subcategory $\filt_{\mathcal{E}}$ of $\rep(\Lambda)$. (The details of our notations can be found in Section \ref{quiverel-sec} and Section \ref{excep-ratio-inv-sec}.) Denote by $Q_{\mathcal{E}}$ the Gabriel quiver of the (smaller) algebra $\Lambda_{\mathcal{E}}$. Consider a dimension vector $\dd'$ of $Q_{\mathcal{E}}$ and set $\dd=\sum_{1 \leq i \leq t}\dd'(i)\ddim E_i$. Now, we can state our next result.

\begin{theorem} \label{rational-inv-quiverel-thm} Keep the same notations as above. Assume that $\rep(\Lambda_{\mathcal{E}},\dd')$ is an irreducible representation variety containing a Schur representation and let $C$ be an irreducible component of $\rep(\Lambda,\dd)$ such that $C \cap \filt_{\mathcal{E}}(\dd) \neq \emptyset$. Then, $k(\rep(\Lambda_{\mathcal{E}},\dd'))^{\GL(\dd')}$ and $k(C)^{\GL(\dd)}$ are isomorphic (as $k$-algebras).
\end{theorem}

Next, we focus on canonical algebras which were discovered and studied by Ringel \cite{R3}. They form a distinguished class of algebras of global dimension two and play an important role in the representation theory of algebras. Moreover, W. Geigle and Lenzing found in \cite{GeiLen} a beautiful interpretation of canonical algebras and their representations in terms of coherent sheaves over weighted projective lines. The invariant theory for canonical algebras in the \emph{regular} case has been investigated in a number of papers, see \cite{Bob1}, \cite{Bob2}, \cite{DL}, \cite{DL2}, \cite{SW2}. By applying Theorem \ref{rational-inv-quiverel-thm} to tame canonical algebras, we are able to describe the fields of rational invariants when the dimension vector in question is not necessarily regular. More precisely, we have:

\begin{theorem} \label{reptype-canonical-thm} Let $\Lambda$ be a canonical algebra. The following conditions are equivalent:
\begin{enumerate}
\renewcommand{\theenumi}{\arabic{enumi}}
\item $\Lambda$ is tame;
\item for each generic root $\dd$ of $\Lambda$ and each indecomposable irreducible component $C$ of $\rep(\Lambda,\dd)$, $k(C)^{\GL(\dd)} \simeq k$ or $k(t)$.
\end{enumerate}
\end{theorem}

Note that the condition on the fields of rational invariants in Theorem \ref{reptype-quivers-thm}{(4)} and Theorem \ref{reptype-canonical-thm}{(2)} simply says that the rational quotients $\rep(Q,\dd)/\GL(\dd)$ and  $C/\GL(\dd)$ are (birationally equivalent to) a point or $\PP^1$ whenever $\dd$ is a generic root, and this is very much in sync with the philosophy behind the tameness of an algebra.

In \cite{Kac2},  V. Kac showed that the problem of computing fields of rational invariants for quivers can be reduced to the case where the dimension vectors involved are Schur roots. In Proposition \ref{rational-inv-generic-decomp-prop}, we explain how to extend this result to fields of rational invariants for arbitrary finite-dimensional algebras. As a direct consequence of Theorem \ref{reptype-quivers-thm}, Theorem \ref{reptype-canonical-thm}, and Proposition \ref{rational-inv-generic-decomp-prop}, we have:

\begin{proposition} \label{ratio-inv-tame-can-prop} Let $\Lambda$ be either a tame path algebra or a tame canonical algebra. If $\dd$ is a dimension vector of $\Lambda$ and $C$ is an irreducible component of $\rep(\Lambda,\dd)$ then
$$
k(C)^{\GL(\dd)}\simeq k(t_1,\ldots,t_N),
$$
where $N$ is the sum of the multiplicities of the isotropic imaginary roots that occur in the generic decomposition of $\dd$ in $C$. 
\end{proposition}

Let us mention that our approach to proving Proposition \ref{ratio-inv-tame-can-prop} when $\Lambda$ is a tame path algebra gives a short and conceptual proof of Ringel's result in \cite{R4}.

The layout of the paper is as follows. In Section \ref{QIT-sec}, we recall some fundamental results from quiver invariant theory. This includes A. King's construction of moduli spaces of quiver representations, and Derksen-Weyman's results on the $\theta$-stable decomposition for dimension vectors of quivers. The proof of Theorem \ref{reptype-quivers-thm} can be found in Section \ref{proof-thm-sec}. In Section \ref{quiverel-sec}, we review some important results about representation varieties and their irreducible components. We prove Theorem \ref{rational-inv-quiverel-thm} in Section \ref{excep-ratio-inv-sec} where we also review important properties of categories of the form $\filt_{\mathcal E}$ which are due to B. Keller, and W. Crawley-Boevey and J. Schr\"oer. In Section \ref{canonical-sec}, we first review some fundamental results about canonical algebras, including a description of the indecomposable irreducible components for tame canonical algebras due to G. Bobi{\'n}ski and Skowro{\'n}ski, and Ch. Geiss and Schr\"oer. Furthermore, we present a systematic approach to finding short orthogonal exceptional sequences of representations via the study of facets of cones of effective weights for quivers with relations; in particular, this requires an extension of the Derksen-Weyman's notion of $\theta$-stable decomposition to quivers with relations. We prove Theorem \ref{reptype-canonical-thm} and Proposition \ref{ratio-inv-tame-can-prop} at the end of this final section.

\section{Quiver invariant theory} \label{QIT-sec} Let $Q=(Q_0,Q_1,t,h)$ be a finite quiver with vertex set $Q_0$ and arrow set $Q_1$. The two functions $t,h:Q_1 \to Q_0$ assign to each arrow $a \in Q_1$ its tail \emph{ta} and head \emph{ha}, respectively.

A representation $V$ of $Q$ over $k$ is a collection $(V(i),V(a))_{i\in Q_0, a\in Q_1}$ of finite-dimensional $k$-vector spaces $V(i)$, $i \in Q_0$, and $k$-linear maps $V(a) \in \Hom_k(V(ta),V(ha))$, $a \in Q_1$. The dimension vector of a representation $V$ of $Q$ is the function $\ddim V : Q_0 \to \ZZ$ defined by $(\ddim V)(i)=\dim_{k} V(i)$ for $i\in Q_0$. Let $S_i$ be the one-dimensional representation of $Q$ at vertex $i \in Q_0$ and let us denote by $\ee_i$ its dimension vector. By a dimension vector of $Q$, we simply mean a function $\dd \in \ZZ_{\geq 0}^{Q_0}$.

Given two representations $V$ and $W$ of $Q$, we define a morphism $\varphi:V \rightarrow W$ to be a collection $(\varphi(i))_{i \in Q_0}$ of $k$-linear maps with $\varphi(i) \in \Hom_k(V(i), W(i))$ for each $i \in Q_0$, and such that $\varphi(ha)V(a)=W(a)\varphi(ta)$ for each $a \in Q_1$. We denote by $\Hom_Q(V,W)$ the $k$-vector space of all morphisms from $V$ to $W$. Let $V$ and $W$ be two representations of $Q$. We say that $V$ is a subrepresentation of $W$ if $V(i)$ is a subspace of $W(i)$ for each $i \in Q_0$ and $V(a)$ is the restriction of $W(a)$ to $V(ta)$ for each $a \in Q_1$. In this way, we obtain the abelian category $\rep(Q)$ of all quiver representations of $Q$.

Given two quiver representations $V$ and $W$, we have the Ringel's \cite{R} canonical exact sequence:
\begin{equation}\label{can-exact-seq}
0 \rightarrow \Hom_Q(V,W) \rightarrow \bigoplus_{i\in
Q_0}\Hom_{k}(V(i),W(i)){\buildrel
d^V_W\over\longrightarrow}\bigoplus_{a\in
Q_1}\Hom_{k}(V(ta),W(ha)),
\end{equation}
where $d^V_W((\varphi(i)_{i\in Q_0})=(\varphi(ha)V(a)-W(a)\varphi(ta))_{a\in Q_1}$ and $\coker(d^V_W)=\Ext^1_{Q}(V,W).$

The Ringel form of $Q$ is the bilinear form $\langle \cdot,\cdot \rangle_{Q} :\ZZ^{Q_0} \times \ZZ^{Q_0}\to \ZZ$ defined by
\begin{equation}\label{Euler-prod}
\langle\dd,\ee \rangle_{Q} = \sum_{i\in Q_0}
d(i)e(i)-\sum_{a \in Q_1} d(ta)e(ha).
\end{equation}
(When no confusion arises, we drop the subscript $Q$.) It follows
from (\ref{can-exact-seq}) and (\ref{Euler-prod}) that
\begin{equation} \label{Euler-formula}
\langle\ddim V, \ddim W \rangle=
\dim_k\Hom_{Q}(V,W)-\dim_k\Ext^1_{Q}(V,W).
\end{equation}

The Tits form of $Q$ is the integral quadratic form $q_{Q}:\ZZ^{Q_0} \to \ZZ$ defined by $q_{Q}(\dd)=\langle \dd,\dd \rangle$ for $\dd \in \ZZ^{Q_0}$.

\subsection{Semi-invariants of quivers} Let $\dd$ be a dimension vector of $Q$. The representation space of $\dd$-dimensional representations of $Q$ is the affine space
$$\rep(Q,\dd)=\prod_{a\in Q_1}\Mat_{\dd(ha) \times \dd(ta)}(k).$$ The group $\GL(\dd)=\prod_{i\in Q_0}\GL(\dd(i),k)$ acts on $\rep(Q,\dd)$ by simultaneous conjugation, i.e., for $g=(g(i))_{i\in Q_0}\in \GL(\dd)$ and $V=(V(a))_{a \in Q_1} \in \rep(Q,\dd)$, we define $g \cdot V$ by $$(g\cdot V)(a)=g(ha)V(a) g(ta)^{-1}, \forall a \in Q_1.$$ In this way, $\rep(Q,\dd)$ becomes a rational representation of the linearly reductive group $\GL(\dd)$ and the $\GL(\dd)-$orbits in $\rep(Q,\dd)$ are in one-to-one correspondence with the isomorphism classes of the $\dd$-dimensional representations of $Q$.

\textbf{From now on, we assume that $Q$ is a quiver without oriented cycles.} Under this assumption, one can show that there is only one closed $\GL(\dd)-$orbit in $\rep(Q,\dd)$, and hence the invariant ring $\text{I}(Q,\dd):= k[\rep(Q,\dd)]^{\GL(\dd)}$ is exactly the base field $k$.

Now, consider the subgroup $\SL(\dd) \subseteq \GL(\dd)$ defined by
$$
\SL(\dd)=\prod_{i \in Q_0}\SL(\dd(i),k).
$$

Although there are only constant $\GL(\dd)-$invariant polynomial functions on $\rep(Q,\dd)$, the action of $\SL(\dd)$ on
$\rep(Q,\dd)$ provides us with a highly non-trivial ring of semi-invariants. Note that any $\theta \in \ZZ^{Q_0}$ defines a rational character $\chi_{\theta}:\GL(\dd) \to k^*$ by $$\chi_{\theta}((g(i))_{i \in Q_0})=\prod_{i \in Q_0}(\det g(i))^{\theta(i)}.$$ In this way, we can identify $\Gamma=\ZZ ^{Q_0}$ with the group $X^\star(\GL(\dd))$ of rational characters of $\GL(\dd)$, assuming that $\dd$ is a
sincere dimension vector. In general, we have only the natural epimorphism $\Gamma \to X^*(\GL(\dd))$. We also refer to the rational characters of $\GL(\dd)$ as (integral) weights of $Q$.

Let us now consider the ring of semi-invariants $\SI(Q,\dd):= k[\rep(Q,\dd)]^{\SL(\dd)}$. As $\SL(\dd)$ is the commutator subgroup of $\GL(\dd)$ and $\GL(\dd)$ is linearly reductive, we have $$\SI(Q,\dd)=\bigoplus_{\theta \in X^\star(\GL(\dd))}\SI(Q,\dd)_{\theta},
$$
where $$\SI(Q,\dd)_{\theta}=\lbrace f \in k[\rep(Q,\dd)] \mid g f= \theta(g)f \text{~for all~}g \in \GL(\dd)\rbrace$$ is
called the space of semi-invariants of weight $\theta$.

If $\dd \in \Gamma$, we define $\theta=\langle \dd,\cdot \rangle$ by
$$\theta(i)=\langle \dd,\ee_i \rangle, \forall i\in
Q_0.$$ Similarly, one can define the weight $\tau=\langle \cdot,\dd \rangle.$

\subsection{Reciprocity and polynomiality properties} \label{non-log-concave-sec} The following remarkable properties of weight spaces of semi-invariants, due to Derksen and Weyman \cite{DW1,DW3}, play a crucial role in the proof of Theorem \ref{reptype-quivers-thm}.

\begin{proposition}[Reciprocity Property]\cite[Corollary 1]{DW1} \label{reciproc-prop} Let $Q$ be a quiver and let $\dd$ and $\ee$ be two dimension vectors of $Q$. Then
$$\dim_k\SI(Q,\ee)_{\langle \dd, \cdot \rangle} = \dim_k
\SI(Q,\dd)_{-\langle \cdot, \ee \rangle}.$$
\end{proposition}

For two dimension vectors $\dd$ and $\ee$, we define
$$
\dd \circ \ee=\dim_k \SI(Q,\ee)_{\langle \dd, \cdot \rangle} = \dim_k \SI(Q,\dd)_{-\langle \cdot, \ee \rangle}.
$$

The next result tells us how the dimensions $(N\dd) \circ \ee$ and $\dd \circ (N\ee)$ vary as $N \in \ZZ_{\geq 0}$ varies.

\begin{proposition}\cite[Corollary 1]{DW3}\label{polynomiality-wtspaces-prop} Let $Q$ be a quiver and let $\dd$ and $\ee$ be two dimension vectors of $Q$ such that $\dd \circ \ee \neq 0$. Then there exist polynomials $P, Q \in \QQ[X]$ (both depending on $\dd$ and
$\ee$) with $P(0)=Q(0)=1$, and
$$
(N\dd) \circ \ee=P(N), \forall N \geq 0,
$$
and
$$
\dd \circ (N\ee)=Q(N), \forall N \geq 0.
$$
\end{proposition}

\begin{remark} Note that Proposition \ref{polynomiality-wtspaces-prop} immediately implies the fact that weight spaces of semi-invariants of quivers are \emph{asymptotically log-concave} in both arguments (see \cite{CDW}). However, they are not log-concave in general.
\end{remark}

\subsection{Moduli spaces of quiver representations}
In \cite{K}, King constructed, via GIT, moduli spaces of representations for finite-dimensional algebras. Let $\dd$ be a dimension vector of $Q$. Then the one-dimensional torus
$$
T_1=\{(t\Id_{\dd(i)})_{i \in Q_0} \mid t \in k^* \} \subseteq
\GL(\dd)
$$
acts trivially on $\rep(Q,\dd)$, and so there is a well-defined action of $\PGL(\dd):={\GL(\dd)/T_1}$ on $\rep(Q,\dd)$.

\begin{definition} \cite[Definition 2.1]{K} Let $\theta \in \ZZ^{Q_0}$ be an integral weight of $Q$. A representation $V \in \rep(Q,\dd)$ is said to be:
\begin{enumerate}
\renewcommand{\theenumi}{\arabic{enumi}}

\item \emph{$\theta$-semi-stable} if there exists a semi-invariant $f \in \SI(Q,\dd)_{n\theta}$ with $n \geq 1$, such that $f(V)\neq 0$;

\item \emph{$\theta$-stable} if there exists a semi-invariant $f \in \SI(Q,\dd)_{n\theta}$ with $n \geq 1$, such that $f(V)\neq 0$ and, furthermore, the $\GL(\dd)$-action on the principal open subset defined by $f$ is closed and $\dim \GL(\dd)V=\dim \PGL(\dd)$.
\end{enumerate}
\end{definition}

Now, consider the (possibly empty) open subsets
$$\rep(Q,\dd)^{ss}_{\theta}=\{V \in \rep(Q,\dd)\mid V \text{~is~}
\text{$\theta$-semi-stable}\}$$
and $$\rep(Q,\dd)^s_{\theta}=\{V \in \rep(Q,\dd)\mid V \text{~is~}
\text{$\theta$-stable}\}$$
of $\dd$-dimensional $\theta$(-semi)-stable representations.

We say that a dimension vector $\dd$ is \emph{$\theta$(-semi)-stable} if there exists $\theta$(-semi)-stable representation $V \in \rep(Q,\dd)$. A dimension vector $\dd$ is called a \emph{Schur root} if there exists a representation $V \in \rep(Q,\dd)$ such that $\End_{Q}(V)=k$; we call such a representation a \emph{Schur representation}. Note that if $\dd$ is $\theta$-stable for some integral weight $\theta$ then $\dd$ is a Schur root.

The GIT-quotient of $\rep(Q,\dd)^{ss}_{\theta}$ by $\PGL(\dd)$ is
$$
\M(Q,\dd)^{ss}_{\theta}:=\Proj(\bigoplus_{n \geq 0}\SI(Q,\dd)_{n\theta}).
$$
This is an irreducible projective variety whose closed points parameterize the closed $\GL(\dd)$-orbits in $\rep(Q,\dd)^{ss}_{\theta}$.

From geometric invariant theory we also know that $\M(Q,\dd)^{ss}_{\theta}$ contains a (possibly empty) open subset $\M(Q,\dd)^s_{\theta}$ which is a geometric quotient of $\rep(Q,\dd)^s_{\theta}$ by $\PGL(\dd)$.

\begin{remark}\label{stable-dim-rationality-rmk} Let $\dd$ be a $\theta$-stable dimension vector where $\theta \in \ZZ^{Q_0}$. It follows from Rosenlicht's theorem \cite{Ros} that $k(\M(Q,\dd)^{ss}_{\theta}) \simeq k(\rep(Q,\dd))^{\GL(\dd)}$. Also, a simple dimension count shows that $\dim \M(Q,\dd)^{ss}_{\theta}=1-\langle \dd,\dd \rangle$.
\end{remark}

\begin{theorem}\cite{S1}\label{S-Schur-Rational-thm} Let $\dd$ be a dimension vector of $Q$. Then $\dd$ is a Schur root if and only if $\dd$ is $\theta_{\dd}$-stable where $\theta_{\dd}=\langle \dd,\cdot\rangle-\langle \cdot,\dd \rangle$.
\end{theorem}

\subsection{The $\theta$-stable decomposition for dimension vectors} In this section, the $\theta$-stable decomposition for dimension vectors, due to Derksen and Weyman \cite{DW2}, is reviewed.

Let $Q$ be a quiver and let $\dd$ be a $\theta$-semi-stable dimension vector of $Q$ where $\theta \in \ZZ^{Q_0}$. (Note that in particular this implies $\theta(\dd):=\sum_{i \in Q_0}\theta(i)\dd(i)=0$.) One of the fundamental results about semi-stable representations is the King's \cite{K} numerical criterion for (semi-)stability that says that a representation $V \in \rep(Q,\dd)$ is $\theta$-semi-stable if and only if $\theta(\ddim V') \leq 0$ for all subrepresentations $V'$ of $V$. Furthermore, $V$ is $\theta$-stable if and only if $\theta(\ddim V')<0$ for all proper subrepresentations $V'$ of $V$.

We define $\rep(Q)^{ss}_{\theta}$ to be the full subcategory of $\rep(Q)$ consisting of all $\theta$-semi-stable representations. Similarly, we denote by $\rep(Q)^{s}_{\theta}$ the full subcategory of $\rep(Q)$ consisting of all $\theta$-stable representations. (Of course, the zero representation is always semi-stable but not stable.)

It is easy to see that $\rep(Q)^{ss}_{\theta}$ is a full exact abelian subcategory of $\rep(Q)$ which is closed under extensions and whose simple objects are precisely the $\theta$-stable representations. Moreover, $\rep(Q)^{ss}_{\theta}$ is Artinian and Noetherian, and hence every $\theta$-semi-stable representation has a Jordan-H{\"o}lder filtration in $\rep(Q)^{ss}_{\theta}$.

Following \cite{DW2}, we say that $$\dd=\dd_1 \pp \dd_2 \pp \ldots \pp \dd_l$$ is the \emph{$\theta$-stable decomposition} of $\dd$ if a general representation in $\rep(Q,\dd)$ has a Jordan-H\"older filtration in $\rep(Q)^{ss}_{\theta}$ with factors of dimensions $\dd_1, \ldots ,\dd_l$ in some order.

Recall that a root of $Q$ is just the dimension vector of an indecomposable representation of $Q$. We say that a root $\dd$ is \emph{real} if $\langle \dd, \dd \rangle=1$. If $\langle \dd, \dd \rangle=0$, $\dd$ is said to be an imaginary isotropic root. Finally, we say that $\dd$ is an imaginary but non-isotropic root if $\langle \dd, \dd \rangle<0$.

In what follows, we write $m\cdot\dd$ instead of $\underbrace{\dd \pp \dd \pp \ldots \pp \dd}_{m}$. The projective scheme $\Proj(\bigoplus_{n\geq 0}S^m(\SI(Q,\dd)_{n\theta}))$ is denoted by $S^m(\M(Q,\dd)^{ss}_{\theta})$. (Here, $S^m(\SI(Q,\dd)_{n\theta})$ is the $m^{th}$ symmetric power of $\SI(Q,\dd)_{n\theta}$.) The following theorem of Derksen and Weyman, plays a crucial role in our study.

\begin{theorem} \cite{DW2} \label{stable-decomp-thm}
Let $\dd=\dd_1 \pp \dd_2 \pp \ldots \pp \dd_l$ be the
$\theta$-stable decomposition of $\dd$ and let $m $ be a positive integer.
\begin{enumerate}
\renewcommand{\theenumi}{\arabic{enumi}}
\item The $\theta$-stable decomposition of $m\dd$ is $m\dd=[m\dd_1]\pp
\ldots \pp[m\dd_l]$, where
$$[m\dd_i]=\begin{cases}
m \cdot \dd_i & \text{if $\dd_i$ is real or
isotropic};\\
m\dd_i & \text{otherwise.}
\end{cases}
$$
\item Suppose that $\dd=m_1 \cdot \dd_1 \pp \ldots \pp m_n \cdot \dd_n$ with $m_i$ positive integers and $\dd_i \neq \dd_j$ for $1 \leq i \neq j \leq n$. Then
$$
\SI(Q,\dd)_{m\theta} \simeq \bigotimes_{i=1}^n S^{m_i}(\SI(Q,\dd_i)_{m\theta})
$$
and
$$
\M(Q,\dd)^{ss}_{\theta} \simeq S^{m_1}(\M(Q,\dd_1)^{ss}_{\theta}) \times \ldots \times S^{m_n}(\M(Q,\dd_n)^{ss}_{\theta}).
$$

\end{enumerate}
\end{theorem}

\begin{remark} \label{stable-decomp-rmk} $(1)$ Note that in particular Theorem \ref{stable-decomp-thm}{(1)} says that if $\dd$ is $\theta$-stable and $\langle \dd, \dd \rangle <0$ then $m\dd$ is still $\theta$-stable for all integers $m \geq 1$.\\
$(2)$ Although the above isomorphism between moduli spaces is not explicitly mentioned in \cite[Theorem 3.20]{DW2}, it follows immediately from its proof.
\end{remark}

\section{Proof of Theorem \ref{reptype-quivers-thm}}\label{proof-thm-sec}

First, let us briefly recall the local quiver technique of  Adriaenssens-Le Bruyn \cite{ALeB}. Let $Q$ be a quiver without oriented cycles and let $\dd$ be a $\theta$-semi-stable dimension vector of $Q$ where $\theta \in \ZZ^{Q_0}$. Let $\pi:\rep(Q,\dd)^{ss}_{\theta}\to \M(Q,\dd)^{ss}_{\theta}$ be the quotient map and let $\xi$ be a closed point of $\M(Q,\dd)^{ss}_{\theta}$. Then, the fiber $\pi^{-1}(\xi)$ contains a unique orbit $\GL(\dd)M$ which is closed in $\rep(Q,\dd)^{ss}_{\theta}$. As shown by King in \cite{K}, this is equivalent to saying that $M=\bigoplus_{i=1}^l M_i^{m_i}$ with $M_1,\ldots,M_l$ pairwise non-isomorphic $\theta$-stable representations (call such a representation $M$ $\theta$-polystable). Next, consider the local quiver setup $(Q_{\xi},\dd_{\xi})$ where $Q_{\xi}$ has vertex set $\{1,\ldots, l\}$ and $\dim_k\Ext^{1}_{Q}(M_i,M_j)$ arrows from vertex $i$ to vertex $j$; the dimension vector $\dd_{\xi}$ of $Q_{\xi}$ is defined to be the vector $(m_1,\ldots, m_l)$. It was proved in \cite{ALeB} that $\M(Q,\dd)^{ss}_{\theta}$ is smooth at $\xi$ if and only if the ring of invariants $k[\rep(Q_{\xi},\dd_{\xi})]^{\GL(\dd_{\xi})}$ is a polynomial algebra. If this is the case, we call $(Q_{\xi},\dd_{\xi})$ a coregular quiver setup.

\begin{lemma}\label{moduli-Euclidean-lemma} Let $Q$ be a Euclidean quiver and let $\delta$ be the unique isotropic Schur root of $Q$. If $\theta \in \ZZ^{Q_0}$ is a weight such that $\delta$ is $\theta$-stable then $\M(Q,\delta)^{ss}_{\theta} \simeq \PP^1$.
\end{lemma}

\begin{remark} This result is well-known. For example, it follows from the work of Domokos and Lenzing on moduli spaces of regular representations for concealed-canonical algebras (see \cite[Corollary 7.3]{DL2}). Here, we give a direct proof working entirely within the category of representations of the tame quiver $Q$. More precisely, we make essential use of: (i) the rationality of $k(\rep(Q,\delta))^{\GL(\delta)}$; (ii) the Derksen-Weyman $\theta$-stable decomposition for dimension vectors; (iii) Adriaenssens-Le Bruyn local quiver technique.
\end{remark}

\begin{proof} First, note that the moduli space $\M(Q,\delta)^{ss}_{\theta}$ is a projective curve since $\langle \delta, \delta \rangle=0$. Moreover, if follows from the work of Ringel \cite{R2} or Schofield \cite{S3}  that $\M(Q,\delta)^{ss}_{\theta}$ is a rational variety. (For another proof of this rationality property, see our Corollary \ref{rational-inv-Euclidean-coro}.) To prove that $\M(Q,\delta)^{ss}_{\theta}$ is precisely $\PP^1$, it remains to show that $\M(Q,\delta)^{ss}_{\theta}$ is smooth (see for example \cite[Exercise I.6.1]{Har}). For this, let $M=\bigoplus_{i=1}^l M_i^{m_i}$ be a $\delta$-dimensional $\theta$-polystable representation with $M_1,\ldots,M_l$ pairwise non-isomorphic $\theta$-stable representations, and let $(Q_M,\dd_M)$ be the corresponding local quiver setup. Of course, if $l=1$ then the ring of invariants $k[\rep(Q_{M},\dd_M)]^{\GL(\dd_M)}$ is just $k[t]$. So, let us assume that $l>1$. Note that since each $\ddim M_i$ is a Schur root smaller (coordinatewise) than $\delta$, $\ddim M_i$ has to be a real Schur root, and hence $M_i$ is an exceptional representation. In particular, $Q_M$ has no loops.

If at least one of the $m_i$, say $m_1$, is bigger than one then we claim that $Q_M$ has no oriented cycles. To see why this is so, consider the $\theta$-polystable representation $M'=M^{m_1-1}\oplus \bigoplus_{i=2}^{l}M_i^{m_i}$ and denote its dimension vector by $\dd'$. Since $\dd'<\delta$, we know that all the Schur roots that occur in the canonical decomposition of $\dd'$ are real, and hence $\GL(\dd')$ acts with a dense orbit on $\rep(Q,\dd')$ (see \cite[Corollary 1]{Kac}). In particular, $M'$ is the unique $\dd'$-dimensional $\theta$-polystable representation, up to isomorphism. Hence, $\dd'=(m_1-1)\cdot \ddim{M_1} \pp \ldots \pp m_l\cdot \ddim{M_l}$ is the $\theta$-stable decomposition of $\dd'$. It now follows from \cite[Proposition 3.18(d)]{DW2} that $Q_M$ has no oriented cycles, and so $k[\rep(Q_M,\dd_M)]^{\GL(\dd_M)}=k$.

Next, let $I \subset \{1,\dots, l\}$ be a proper subset. Denote by $\dd_I=\sum_{i \in I}m_i\ddim M_i$ and note that $\GL(\dd_I)$ acts with a dense orbit on $\rep(Q,\dd_I)$ as $\dd_I<\delta$. Arguing as before we deduce that the (full) subquiver of $Q_M$ with vertex set $I$ has no oriented cycles.

From the discussion above, it remains to look into the case when $Q_M$ has oriented cycles,  $m_i=1$ for all $1 \leq i \leq l$,  and any oriented cycle in $Q_M$ uses all the vertices. So, we can reorder the vertices of $Q_M$, if needed, so that if $k_{i,j}$ denotes the number of arrows from $i$ to $j$, $1 \leq i,j\leq l$, then $k_{1,2}\cdot k_{2,3}\cdot \ldots \cdot k_{l-1,l}\cdot k_{l,1}\neq 0$, and the rest of the $k_{i,j}$ are zero. Furthermore, we have that $k_{1,2}+\ldots+k_{l,1}=l$ as $\langle \delta,\delta \rangle=0$; in other words, $Q_M$ is just the oriented $l$-cycle. But for this quiver and the thin sincere dimension vector $\dd_M$, the corresponding ring of invariants is known to be a polynomial algebra (see for example \cite{Bock}).

Finally, using the local quiver technique and the fact that all local quiver setups associated to $\M(Q,\delta)^{ss}_{\theta}$ are coregular, we conclude that $\M(Q,\delta)^{ss}_{\theta}$ is smooth. This finishes the proof.
\end{proof}

Now, we are ready to prove Theorem \ref{reptype-quivers-thm}.

\begin{proof}[Proof of Theorem \ref{reptype-quivers-thm}] First, let us prove the implication $(1) \Longrightarrow (2)$. If $Q$ is a Dynkin quiver, $\M(Q,\dd)^{ss}_{\theta}$ is just a point as $\GL(\dd)$ acts with a dense orbit on $\rep(Q,\dd)$. Next, let assume that $Q$ is a Euclidean quiver and let $\dd$ be a $\theta$-semi-stable dimension vector where $\theta \in \ZZ^{Q_0}$. If the isotropic Schur root $\delta$ of $Q$ does not occur in the $\theta$-stable decomposition of $\dd$ then Theorem \ref{stable-decomp-thm}{(2)} tells us that $\M(Q,\dd)^{ss}_{\theta}$ is just a point. Otherwise, let $m$ be the multiplicity of $\delta$ in the $\theta$-stable decomposition of $\dd$. It follows again from Theorem \ref{stable-decomp-thm}{(2)} that
$$
\M(Q,\dd)^{ss}_{\theta} \simeq S^m(\M(Q,\delta)^{ss}_{\theta}).
$$
Furthermore,  $\M(Q,\delta)^{ss}_{\theta}\simeq \PP^1$ by Lemma \ref{moduli-Euclidean-lemma}, and so we deduce that $\M(Q,\dd)^{ss}_{\theta}\simeq \PP^m$.

Next, we prove the implication $(2)\Longrightarrow (3)$. Let $\dd$ be a $\theta$-semi-stable dimension vector where $\theta$ is an integral weight. Then, we know that $\M(Q,\dd)^{ss}_{\theta} \simeq \PP^m \hookrightarrow \PP^r$. Choose an integer $l \geq 1$ for which the graded algebra $\bigoplus_{n \geq 0}\SI(Q,\dd)_{n(l\theta)}$ is generated by the degree one component $\SI(Q,\dd)_{l\theta}$. Pulling back the line bundle $\mathcal{O}(1)$ over $\PP^r$, we get a line bundle over $\M(Q,\dd)^{ss}_{\theta}$ which has to be of the form $\mathcal{O}(d)$. The image of the map $\Gamma(\PP^r,\mathcal{O}(N)) \to \Gamma(\PP^m,\mathcal{O}(Nd))$ is precisely $\SI(Q,\dd)_{N(l\theta)}$, and furthermore this map is surjective for sufficiently large values of $N$ (see for example \cite[Exercise II.5.9]{Har}). Hence,
$$
\dim_k \SI(Q,\dd)_{N(l\theta)}=\binom{Nd+m}{m}
$$
for sufficiently large values of $N$. From Proposition \ref{polynomiality-wtspaces-prop}, we deduce that
$$
\dim_k \SI(Q,\dd)_{n\theta}=\binom{qn+m}{m},
$$
for all integers $n \geq 0$ where $q=\frac{d}{l}$. This clearly shows that the weight spaces of semi-invariants of $Q$ are log-concave.

To prove that $(3) \Longrightarrow (1)$, we follow the arguments in \cite[Sec. 3.4]{CDW}. Let us assume to the contrary that $Q$ is a wild quiver. Under this assumption, we can always find a non-isotropic imaginary Schur root $\dd'$. Then $\dd'$ is stable with respect to the weight $\theta_{\dd'}=\langle \dd', \cdot \rangle-\langle \cdot, \dd' \rangle$ by Theorem \ref{S-Schur-Rational-thm}. We can clearly assume that $\dd'$ is sincere since otherwise we can just simply work with the full subquiver of $Q$ whose vertex set consists of those vertices $i \in Q_0$ for which $\dd'(i)>0$. So, we can write $\theta_{\dd'}=\langle \dd'',\cdot \rangle$ for a unique dimension vector $\dd''$ of $Q$ due to \cite[Lemma 6.5.7]{IOTW} (see also \cite[Theorem 1]{DW1}).

By Remark \ref{stable-decomp-rmk}, we know that $m\dd'$ is still $\langle \dd'', \cdot \rangle$-stable for any integer $m \geq 1$. Consequently, the dimension of the moduli space $\M(Q,m\dd')^{ss}_{\langle \dd'', \cdot \rangle}$ is $1-m^2\langle \dd',\dd'\rangle$. So, for any integer $m \geq 1$, the Hilbert function $(N\dd'') \circ (m\dd')$ is a polynomial in $N$ of degree $1-m^2 \langle \dd', \dd' \rangle$. As $-\langle \dd', \dd' \rangle \geq 1$, we have that for sufficiently large $N$,
$$
(N \dd'') \circ (2 \dd')>((N \dd'') \circ \dd')^2,
$$
which is equivalent to
$$
\dim_k \SI(Q,\dd)_{2\theta}>(\dim_k \SI(Q,\dd)_{\theta})^2,
$$
where $\dd=N \dd''$ and $\theta=-\langle \cdot,\dd' \rangle$. But this is a contradiction. So, $Q$ must be a tame quiver.

Finally, it remains to prove the equivalence $(1) \Longleftrightarrow (4)$. First, let us assume that for each Schur root $\dd$, the field of rational invariants $k(\rep(Q,\dd))^{\GL(\dd)}$ is isomorphic to $k$ or $k(t)$. For any dimension vector $\dd$ of $Q$, we have $$\trdeg_{k} k(\rep(Q,\dd))^{\GL(\dd)}=\dim \rep(Q,\dd)- \dim \GL(\dd)+\min \{\dim_k \End_Q(V) \mid V \in \rep(Q,\dd) \}.$$
(This formula follows immediately from Rosenlicht's theorem \cite[Theorem 2]{Ros} and the fiber dimension theorem \cite{Sha}.) So, if $\dd$ is a Schur root of $Q$ then it is easy to see that $q_{Q}(\dd) \geq 0$. Now, let $\dd$ be a dimension vector of $Q$ and consider its canonical decomposition
$$
\dd=\dd_1 \oplus \ldots \oplus \dd_m,
$$
where $\dd_i$, $1 \leq i \leq m$, are Schur roots and $\ext^1_Q(\dd_i,\dd_j)=0$, $\forall 1 \leq i \neq j \leq m$. In particular, we have $q_{Q}(\dd) \geq \sum _{1 \leq i \leq m} q_{Q}(\dd_i) \geq 0$. This implies that $Q$ is a Dynkin or Euclidean quiver.

The implication $(1) \Longrightarrow (4)$ follows from Corollary \ref{rational-inv-Euclidean-coro} or \cite{R4}.
\end{proof}

\section{Quivers with relations} \label{quiverel-sec}
Given a quiver $Q$, its path algebra $kQ$ has a $k$-basis consisting of all paths and the multiplication in $kQ$ is given by concatenation of paths. It is easy to see that any finite-dimensional left $kQ$-module defines a representation of $Q$, and vice-versa. Furthermore, the category $\module(kQ)$ of finite-dimensional left $kQ$-modules is equivalent to the category $\rep(Q)$. In what follows, we identify $\module(kQ)$ and $\rep(Q)$, and use the same notation for a module and the corresponding representation.

A \emph{relation} in $Q$ with coefficients in $k$ is an element $r \in kQ$ of the form
$$
r=\sum_{i=1}^l \lambda_i p_i,
$$
where $\lambda_1, \ldots, \lambda_l \in k$ are non-zero scalars and $p_1, \ldots, p_l$ are paths in $kQ$ of length at least two with $tp_1=\dots=tp_l$ and $hp_1=\dots=h p_l$.

A set $R$ of relations is said to be \emph{minimal} if for every $r \in R$, $r$ does not belong to the two-sided ideal $\langle R\setminus \{r\} \rangle$ of $kQ$ generated by $R\setminus \{r\} $. A bound quiver consists of a quiver $Q$ and a minimal finite set $R$ of relations such that there exists a positive integer $L$ with the property that any path in $Q$ of length at least $L$ belongs to the two sided ideal $\langle R \rangle$ of $kQ$ generated by $R$. We call $kQ/\langle R \rangle$ the \emph{bound quiver algebra} of the bound quiver $(Q,R)$. A representation $M$ of $kQ/\langle R \rangle$ (or $(Q,R)$) is just a representation $M$ of $Q$ such that $M(r)=0$ for all $r \in R$.

It is well-known that any finite-dimensional basic algebra $\Lambda$ is isomorphic to the bound quiver algebra of a bound quiver $(Q_{\Lambda},R)$, where $Q_{\Lambda}$ is the Gabriel quiver of $\Lambda$. Note that the set of relations $R$ is not uniquely determined by $\Lambda$. We say that $\Lambda$ is a \emph{triangular} algebra if its Gabriel quiver has no oriented cycles.

Fix a bound quiver $(Q,R)$ and let $\Lambda=kQ/\langle R \rangle$ be its bound quiver algebra. The category $\module(\Lambda)$ of finite-dimensional left $\Lambda$-modules is equivalent to the category $\rep(\Lambda)$ of representations of $\Lambda$. As before, we identify $\module(\Lambda)$ and $\rep(\Lambda)$, and make no distinction between $\Lambda$-modules and representations of $\Lambda$. By a $\Lambda$-module, we always mean a finite-dimensional left $\Lambda$-module. For each vertex $v \in Q_0$, we denote by $e_v$ the primitive idempotent corresponding to $v$.

\subsection{Representation varieties and the Tits form}
Let $\dd$ be a dimension vector of $\Lambda$ (or equivalently, of $Q$). The variety of $\dd$-dimensional representations of $\Lambda$ is the affine variety
$$
\rep(\Lambda,\dd)=\{M \in \rep(Q, \dd) \mid M(r)=0, \forall r \in
R\}.
$$
It is clear that $\rep(\Lambda,\dd)$ is a $\GL(\dd)$-invariant closed subset of $\rep(Q,\dd)$. Note that $\rep(\Lambda, \dd)$ does not have to be irreducible. We call $\rep(\Lambda,\dd)$ the \emph{representation/module variety} of $\dd$-dimensional representations/modules of $\Lambda$.

In what follows, we list a series of important results describing the representation type of a (triangular) algebra $\Lambda$ in terms of the so-called Tits form of $\Lambda$.

\begin{proposition}\cite{delaP}\label{delaPena-tame-irr-prop} If  $\Lambda$ is a tame algebra then
$$
\dim \GL(\dd) \geq \dim \rep(\Lambda,\dd),
$$
for each dimension vector $\dd$ of $\Lambda$.
\end{proposition}

Assume form now on that $\Lambda$ has finite global dimension; this happens, for example, when $Q$ has no oriented cycles. The Ringel form of $\Lambda$ is the bilinear form  $\langle \cdot, \cdot \rangle_{\Lambda} : \ZZ^{Q_0}\times \ZZ^{Q_0} \to \ZZ$ defined by
$$
\langle \dd,\ee \rangle_{\Lambda}=\sum_{l\geq 0}(-1)^l \sum_{i,j\in Q_0}\dim_k \Ext^l_{\Lambda}(S_i,S_j)\dd(i)\ee(j).
$$
Note that if $M$ is a $\dd$-dimensional representation of $\Lambda$ and $N$ is an $\ee$-dimensional representation of $\Lambda$ then
$$
\langle \dd,\ee \rangle_{\Lambda}=\sum_{l\geq 0}(-1)^l \dim_k \Ext^l_{\Lambda}(M,N).
$$
The quadratic form induced by $\langle \cdot,\cdot \rangle_{\Lambda}$ is denoted by $\chi_{\Lambda}$.

The \emph{Tits form} of $\Lambda$ is the integral quadratic form $q_{\Lambda}: \ZZ^{Q_0} \to \ZZ$ defined by
$$q_{\Lambda}(\dd):=\sum_{i \in Q_0}\dd^2(i)-\sum_{i,j\in Q_0}\dim_{k}\Ext^1_{\Lambda}(S_i,S_j)\dd(i)\dd(j)+\sum_{i,j\in Q_0}\dim_{k}\Ext^2_{\Lambda}(S_i,S_j)\dd(i)\dd(j).$$

Next, let us assume that $\Lambda$ is triangular. Under this assumption, the Tits form $q_{\Lambda}$ is related to the geometry of the varieties of representations of $\Lambda$ in the following way. First, $r(i,j):=|R \cap e_j\langle R \rangle e_i|$ is precisely $\dim_{k}\Ext^2_{\Lambda}(S_i,S_j), \forall i,j \in Q_0$, as shown by K. Bongartz in \cite{Bon}. So, we can write
$$
q_{\Lambda}(\dd)=\sum_{i \in Q_0}\dd^2(i)-\sum_{a\in Q_1}\dd(ta)\dd(ha)+\sum_{i,j\in Q_0}r(i,j)\dd(i)\dd(j).
$$

Now, let $\dd$ be a dimension vector of $\Lambda$ and $M \in \rep(\Lambda,\dd)$. By Krull's Principal Ideal Theorem (see for example \cite{Eisbook}), we have
$$
\dim_{M}\rep(\Lambda,\dd)\geq \sum_{a \in Q_1}\dd(ta)\dd(ha)-\sum_{i,j\in Q_0}r(i,j)\dd(i)\dd(j).
$$
In particular, we have that
\begin{equation}\label{Tits-form-ineq}
q_{\Lambda}(\dd)\geq \dim \GL(\dd)-\dim \rep(\Lambda,\dd).
\end{equation}

This inequality together with Proposition \ref{delaPena-tame-irr-prop} proves:

\begin{theorem} \cite{delaP} If $\Lambda$ is a tame triangular algebra then $q_{\Lambda}(\dd)\geq 0$ for each dimension vector $\dd$ of $\Lambda$.
\end{theorem}

We should point out the the converse of this theorem is false in general (see \cite[Example 2.1]{BobRieSko}). However, there are important classes of finite-dimensional algebras for which the converse holds true.

%Talk about quasi-tilted algebras (this includes the canonical algebras) and strongly simply-connected algebras.

\begin{theorem}\cite{BS1, BruPenSko}\label{q-tilted-and-strongly-s-c-Tits-thm} Let $\Lambda$ be either a quasi-tilted algebra or a strongly simply-connected algebra. Then, $\Lambda$ is tame if and only if the Tits form $q_{\Lambda}$ is weakly positive semi-definite.
\end{theorem}

\subsection{The generic decomposition for irreducible components} Let $\Lambda$ be the bound quiver algebra of a bound quiver $(Q,R)$, $\dd$ a dimension vector of $\Lambda$, and $C$ an irreducible component of $\rep(\Lambda, \dd)$.

We say that $C$ is an \emph{indecomposable} irreducible component if $C$ has a non-empty open subset of indecomposable representations. We call $C$ a \emph{Schur} irreducible component if $C$ contains a Schur representation. Note that a Schur irreducible component is always indecomposable. The converse is always true for finite-dimensional path algebras. Finally, we say that $\dd$ is a \emph{generic root} of $\Lambda$ if $\rep(\Lambda,\dd)$ has an indecomposable irreducible component.

Now, let us consider a decomposition $\dd=\dd_1+\ldots +\dd_t$ where $\dd_i \in \ZZ^{Q_0}_{\geq 0}, 1 \leq i \leq t$. %Fix a decomposition $k^{\dd(v)}=\bigoplus_{i=1}^t k^{\dd_i(v)}$ for each vertex $v \in Q_0$ and then consider the induced diagonal embedding $\rep(\Lambda,\dd_1)\times \ldots \times \rep(\Lambda,\dd_t) \hookrightarrow \rep(\Lambda,\dd) $. If $C_i$ is a $\GL(\dd_i)$-invariant subset of $\rep(\Lambda,\dd_i)$, $1 \leq i \leq t$, we denote by $\overline{C_1\oplus \ldots \oplus C_t}$ the closure in $\rep(\Lambda,\dd)$ of the $\GL(\dd)$-orbit of the image of $C_1\times \ldots \times C_t$ through the diagonal embedding.
 If $C_i$ is a $\GL(\dd_i)$-invariant subset of $\rep(\Lambda,\dd_i)$, $1 \leq i \leq t$, we denote by $C_1\oplus \ldots \oplus C_t$ the constructible subset of $\rep(\Lambda,\dd)$ consisting of all modules isomorphic to direct sums of the form $\bigoplus_{i=1}^t X_i$ with $X_i \in C_i, \forall 1 \leq i \leq t$.

The following fundamental result, which defines the generic decomposition for irreducible components in representation varieties, is due to J. A. de la Pe{\~n}a \cite[Section 1.3]{delaP} and Crawley-Boevey-Schr\"oer \cite[Theorem 1.1]{C-BS}.

\begin{theorem} If $C$ is an irreducible component of $\rep(\Lambda,\dd)$ then there are unique generic roots $\dd_1, \ldots, \dd_t$ of $\Lambda$ such that $\dd=\dd_1+\ldots +\dd_t$ and
$$
C=\overline{C_1\oplus \ldots \oplus C_t}
$$
for some indecomposable irreducible components $C_i$ of $\rep(\Lambda,\dd_i), 1 \leq i \leq t$. Moreover, the indecomposable irreducible components $C_i, 1 \leq i \leq t,$ are uniquely determined by this property.
\end{theorem}

With the notations of the theorem above, we call $\dd=\dd_1\oplus \ldots \oplus \dd_t$ the generic decomposition of $\dd$ in $C$, and $C=\overline{C_1\oplus \ldots \oplus C_t}$ the generic decomposition of $C$.

Let us record the following useful lemma:

\begin{lemma}\label{generic-decomp-ineq}\cite[Lemma 1.3]{delaP} Let $C$ be an irreducible component in $\rep(\Lambda,\dd)$ and let $C=\overline{C_1\oplus \ldots \oplus C_t}$ be its generic decomposition where $C_i \subseteq \rep(\Lambda,\dd_i)$, $1 \leq i \leq t$, are indecomposable irreducible components. Then,
$$
\dim \GL(\dd)-\dim C \geq \sum_{i=1}^t(\dim \GL(\dd_i)-C_i).
$$
\end{lemma}

Now, we are ready to prove:

\begin{proposition}\label{qt-ssc-prop} Assume that the field of rational invariants $k(C)^{\GL(\dd)}\simeq k$ or $k(t)$ for each generic root $\dd$ of $\Lambda$ and each indecomposable irreducible component $C$ of $\rep(\Lambda,\dd)$. Then, the following statements hold true.
\begin{enumerate}
\renewcommand{\theenumi}{\arabic{enumi}}
\item If $\Lambda$ is a triangular algebra  then $q_{\Lambda}(\dd) \geq 0$ for each dimension vector $\dd$ of $\Lambda$.
\item If $\Lambda$ is a quasi-tilted algebra or a strongly simply-connected algebra then $\Lambda$ is tame.
\end{enumerate}
\end{proposition}

\begin{proof}(1) Let $\dd$ be a generic root of $\Lambda$ and let $C$ be an indecomposable irreducible component of $\rep(\Lambda,\dd)$. Choose $M$ in $C$ so that $\dim_k \End_{\Lambda}(M)=\min \{\dim_k \End_{\Lambda}(M')\mid M' \in C\}$. Then, we have that
$$
tr.deg_k k(C)^{\GL(\dd)}=\dim C-\dim \GL(\dd)+\dim_k\End_{\Lambda}(M),
$$
and so, $\dim \GL(\dd)-\dim C\geq 0$. But this remains true for any dimension vector $\dd$ of $\Lambda$ and any irreducible component $C$ of $\rep(\Lambda,\dd)$ by Lemma \ref{generic-decomp-ineq}. From this and inequality (\ref{Tits-form-ineq}), we deduce that $q_{\Lambda}(\dd)\geq 0$ for each dimension vector $\dd$.

(2) This follows now from the first part and Theorem \ref{q-tilted-and-strongly-s-c-Tits-thm}.
\end{proof}

In what follows, we explain how to reduce the problem of describing fields of rational invariants on irreducible components of representation varieties to the case where the irreducible components involved are indecomposable. This was already done by Kac in \cite{Kac} in the context of quivers with no relations. It turns out that Kac's proof in \emph{loc. cit.} can be extended to arbitrary finite-dimensional algebras. Indeed, let $\Lambda=kQ/\langle R \rangle$ be the bound quiver algebra of a bound quiver $(Q,R)$, $\dd$ a dimension vector, and $C$ an irreducible component of $\rep(\Lambda,\dd)$.

Let $\dd=\dd_1^{\oplus m_1}\oplus \ldots \oplus \dd_n^{\oplus m_n}$ be the generic decomposition of $\dd$ in $C$ where $\dd_1, \dots, \dd_n$ are distinct generic roots of $\Lambda$, and $m_1, \ldots, m_n$ are positive integers. Next, we assume that the generic decomposition of $C$ is of the form
$$C=\overline{C_1^{\oplus m_1}\oplus \ldots \oplus C_n^{\oplus m_n}},$$
where $C_i \subseteq \rep(\Lambda,\dd_i)$, $1 \leq i \leq n$, are indecomposable irreducible components. Fix a decomposition $k^{\dd(v)}=\underbrace{k^{\dd_1(v)}\oplus \ldots \oplus k^{\dd_1(v)}}_{m_1}  \oplus \ldots \oplus \underbrace{k^{\dd_n(v)}\oplus \ldots \oplus k^{\dd_n(v)}}_{m_n}$ for each vertex $v \in Q_0$, and then embed $\widetilde{C}:=C_1^{m_1} \times \ldots \times C_n^{m_n}$ diagonally in $C$ and $G:=\GL(\dd_1)^{m_1} \times \ldots \times \GL(\dd_n)^{m_n} $ diagonally in $\GL(\dd)$.

Denote by $T_1^{(i)}$ the one-dimensional torus in $\GL(\dd_i)$, set $T:=(T_1^{(1)})^{m_1}\times \ldots \times (T_1^{(n)})^{m_n}$, and note that $T \subseteq C_{\GL(\dd)}(M)$ for any $M \in \widetilde{C}$. Next, we choose an open and dense subset $C_{i,0}$ of $C_i$ consisting of indecomposable representations such that $$\overline{\GL(\dd)C_{1,0}^{m_1}\times \ldots \times C_{n,0}^{m_n}}=C.$$Furthermore, for any representation $M \in C_{1,0}^{m_1}\times \ldots \times C_{n,0}^{m_n}$, it is not difficult to see that $T$ is a maximal torus of $C_{\GL(\dd)}(M)$ (see for example \cite[Section 2.2]{KraRie}). Also, note that the centralizer of $T$ in $\GL(\dd)$ is precisely $G$, and the normalizer $N$ of $T$ in $\GL(\dd)$ is
$$N=(\GL(\dd_1)^{m_1}\rtimes S_{m_1}) \times \ldots \times (\GL(\dd_n)^{m_n}\rtimes S_{m_n}).$$ (Here, $S_m$ denotes the symmetric group on $m$ elements.)

Let us summarize what we have obtained so far:
\begin{enumerate}
\renewcommand{\theenumi}{\arabic{enumi}}
\item $\overline{\GL(\dd)\widetilde{C}}=C$;
\item the generic point $M$ in $\widetilde{C}$ has the property that $T$ is a maximal torus of $C_{\GL(\dd)}(M)$; \item $\widetilde{C}$ is an $N$-invariant closed subvariety of $C$.
\end{enumerate}

In what follows, if $R$ is an integral domain, we denote its field of fractions by $\qq(R)$. Moreover, if $K/k$ is a field extension and $m$ is a positive integer, we define $S^m(K/k)$ to be the field $(\qq(K^{\otimes m}))^{S_m}$ which is, in fact, the same as $\qq((K^{\otimes m})^{S_m})$ since $S_m$ is a finite group.

\begin{proposition}\label{rational-inv-generic-decomp-prop} Let $C=\overline{C_1^{\oplus m_1}\oplus \ldots \oplus C_n^{\oplus m_n}}$ be the generic decomposition of $C$ where $C_i \subseteq \rep(\Lambda,\dd_i)$, $1 \leq i \leq n$, are indecomposable irreducible components, $m_1,\ldots, m_n$ are positive integers, and $\dd_i\neq \dd_j, \forall 1 \leq i\neq j\leq n$. Then,
$$
k(C)^{\GL(\dd)} \simeq \qq(\bigotimes_{i=1}^n S^{m_i}(k(C_i)^{\GL(\dd_i)}/k)).
$$
\end{proposition}

\begin{proof} Let $\pi: C \dashrightarrow C/\GL(\dd)$ be the rational quotient map for the action of $\GL(\dd)$ on $C$. Now, property $(1)$ above tells us that the restriction $\phi=\pi| \widetilde{C}:\widetilde{C}\dashrightarrow C/\GL(\dd)$ is a well-defined dominant rational map.

Let $M_0$ be a generic point in $\widetilde{C}$, $M \in \phi^{-1}(\phi(M_0))$, and $g \in \GL(\dd)$ such that $M=gM_0$. Note that $T$ and $g^{-1}Tg$ are maximal tori of $C_{\GL(\dd)}(M_0)$, and so $T=(gg')^{-1}T(gg')$ for some $g' \in C_{\GL(\dd)}(M_0)$. Hence, $g_0:=gg' \in N$ and $M=g_0M_0$, i.e., $NM_0=\phi^{-1}(\phi(M_0))$. It now follows from the universal property of rational quotients that $\phi$ is the rational quotient map for the action of $N$ on $\widetilde{C}$ and hence $k(\widetilde{C})^{N}\simeq k(C)^{\GL(\dd)}$. The proof now follows.
\end{proof}

\begin{remark}\label{rationality-rmk} Let $\Lambda$ be a tame finite-dimensional $k$-algebra, $\dd$ a Schur root of $\Lambda$, and $C$ a Schur irreducible component of $\rep(\Lambda,\dd)$. Using Proposition \ref{delaPena-tame-irr-prop} it is easy to see that $\trdeg_k k(C)^{\GL(\dd)} \in \{0,1\}$, and hence
$$
k(C)^{\GL(\dd)} \simeq k \text{~or~} k(t) \Longleftrightarrow k(C)^{\GL(\dd)} \text{~is a rational field over~} k.
$$

We refer to the problem that asks to prove that $k(C)^{\GL(\dd)}$ is rational over $k$ for each Schur root $\dd$ of $\Lambda$ and each Schur irreducible component $C$ of $\rep(\Lambda,\dd)$ as \emph{the rationality problem for $\Lambda$}. We have seen that the rationality problem for tame quivers is already settled. Moreover, Schofield has obtained in \cite{S3} a birational classification of moduli spaces of representations for quivers. In particular, he solved the rationality problem for quivers when the dimension vectors involved are indivisible Schur roots. However, the rationality problem for wild quivers in the non-indivisible case is a long-standing open problem (see for example \cite{LeB} and \cite{For}). 

To tackle the rationality problem for finite-dimensional algebras, we are going to use homological algebra. This strategy is explained in the next section.  
\end{remark}

\section{Exceptional sequences and rational invariants} \label{excep-ratio-inv-sec} In this section we explain how exceptional sequences can be used in the study of the fields of rational invariants for finite-dimensional algebras.

Let $\Lambda$ be the bound quiver algebra of a bound quiver $(Q,R)$. A sequence $\mathcal{E}=(E_1, \dots, E_t)$ of finite-dimensional $\Lambda$-modules is called an \emph{orthogonal exceptional sequence} if the following conditions are satisfied:
\begin{enumerate}
\renewcommand{\theenumi}{\arabic{enumi}}
\item $E_i$ is an exceptional module, i.e, $\End_{\Lambda}(E_i)=k$ and $\Ext^l_{\Lambda}(E_i,E_i)=0$ for all $l \geq 1$ and $1 \leq i \leq t$;

\item $\Ext_{\Lambda}^l(E_i,E_j)=0$ for all $l \geq 0$  and $1 \leq i<j \leq t$;

\item $\Hom_{\Lambda}(E_j,E_i)=0$ for all $1 \leq i<j \leq t$.
\end{enumerate}
(If we drop condition $(3)$, we simply call $\mathcal{E}$ an \emph{exceptional sequence}.)

Given an orthogonal exceptional sequence $\mathcal{E}$, consider the full subcategory $\filt_{\mathcal{E}}$ of $\rep(\Lambda)$ whose objects $M$ have a finite filtration $0=M_0\subseteq M_1 \subseteq \dots \subseteq M_s=M$ of submodules such that each factor $M_j/M_{j-1}$ is isomorphic to one the $E_i$. It is clear that $\filt_{\mathcal{E}}$ is a full exact subcategory of $\rep(\Lambda)$ which is closed under extensions. Moreover, Ringel \cite{R2} (see also \cite{DW2}) showed that $\filt_{\mathcal{E}}$ is an abelian subcategory whose simple objects are precisely $E_1, \ldots, E_t$.

The category $\filt_{\mathcal{E}}$ is determined by the Yoneda algebra $\Ext^{\bullet}_Q(\bigoplus_{i=1}^t E_i, \bigoplus_{i=1}^t E_i)$ equipped with its (minimal) $A_{\infty}$-algebra structure as shown by Keller \cite{Kel1, Kel2}. More precisely, let us write $\Ext_{\Lambda}^l(\bigoplus_{i=1}^t E_i,\bigoplus_{i=1}^t E_i)=\bigoplus_{i,j} \Ext_{\Lambda}^l(E_j,E_i)$ and consider the induced  $R$-bimodule structure on $\Ext_{\Lambda}^l(\bigoplus_{i=1}^t E_i,\bigoplus_{i=1}^t E_i)$ where $R$ is the commutative $k$-algebra $k^t$. It is clear that each multiplication map $m_n$ of the $A_{\infty}$-algebra $\Ext^{\bullet}_{\Lambda}(\bigoplus_{i=1}^t E_i,\bigoplus_{i=1}^t E_i)$ defines an $R$-bimodule map $$m_n:\Ext_{\Lambda}^1(\bigoplus_{i=1}^t E_i,\bigoplus_{i=1}^t E_i)^{\bigotimes^n_{R}} \to \Ext_{\Lambda}^2(\bigoplus_{i=1}^t E_i,\bigoplus_{i=1}^t E_i).$$

%Clearly, ``the'' multiplication maps $m_n$ of the $A_{\infty}$-algebra $\Ext^{\bullet}_{\Lambda}(\bigoplus_{i=1}^r E_i,\bigoplus_{i=1}^r E_i)$ are compatible with the $R$-bimodule structure of $\Ext_{\Lambda}^{\bullet}(\bigoplus_{i=1}^r E_i,\bigoplus_{i=1}^r E_i)$.

Now, let $Q_{\mathcal{E}}$ be the quiver with vertices $1, \dots, t$ and $\dim_{k}\Ext^1_{\Lambda}(E_i,E_j)^*$ arrows from $i$ to $j$. The Maurer-Cartan map is, by definition, the map $m=\bigoplus_{n\geq 2} m_n$, and hence its dual is $m^*:\Ext^{2}_{\Lambda}(\bigoplus_{i=1}^t E_i,\bigoplus_{i=1}^t E_i)^* \to kQ_{\mathcal{E}}.$ Note that $Q_{\mathcal{E}}$ has no oriented cycles and that is why for the dual of the Maurer-Cartan map we can just work with the path algebra of $Q_{\mathcal{E}}$ instead of its completed path algebra. Also, the two-sided ideal of $kQ_{\mathcal{E}}$ generated by $\Ima(m^*)$ is an admissible ideal and, hence, is generated by finitely many relations. Finally, we define $\Lambda_{\mathcal{E}}=kQ_{\mathcal{E}}/(\Ima(m^*))$.

Now, we are ready to state the following important result:

\begin{theorem}\label{A-infinity-functor-open-thm} Keeping the same notations as above, the following statements are true.
\begin{enumerate}
\renewcommand{\theenumi}{\arabic{enumi}}

\item There exists an equivalence of categories $F_{\mathcal{E}}$ from $\rep(\Lambda_{\mathcal{E}})$ to $\filt_{\mathcal{E}}$ sending the simple $\Lambda_{\mathcal{E}}$-module $S_i$ at vertex $i$ to $E_i$ for all $1 \leq i \leq t$.

\item Given a dimension vector $\dd \in \ZZ_{\geq 0}^{Q_0}$, the set
$$
\filt_{\mathcal{E}}(\dd):=\{M \in \rep(\Lambda,\dd) \mid M \text{~is in~} \filt_{\mathcal{E}}\}
$$
is open in $\rep(\Lambda, \dd)$.
\end{enumerate}
\end{theorem}

The first part of the theorem above is due to Keller \cite{Kel1, Kel2} and uses the $A_{\infty}$-formalism. The second part was proved by Crawley-Boevey and Schr{\"o}er in \cite[Corollary 1.5]{C-BS}.

Let $\mathcal{E}=(E_1, \dots, E_t)$ be an orthogonal exceptional sequence of $\Lambda$-modules and let $F_{\mathcal{E}}$ be a functor as in Theorem \ref{A-infinity-functor-open-thm}{(1)}. Now, consider a dimension vector $\dd'$ of $Q_{\mathcal{E}}$ and set $\dd=\sum_{1 \leq i \leq t}\dd'(i)\ddim E_i$. Next, we explain how the functor $F_{\mathcal{E}}$ gives rise to well-behaved morphisms at the level of representation varieties:

\begin{proposition} \label{morphisms-prop} Keep the same notations as above. Then, there exist a morphism of algebraic groups $\varphi:\GL(\dd')\to \GL(\dd)$ and a regular morphism $f_{\mathcal{E}}:\rep(\Lambda_{\mathcal{E}},\dd') \to \rep(\Lambda, \dd)$ such that:
\begin{enumerate}
\renewcommand{\theenumi}{\arabic{enumi}}
\item $f_{\mathcal{E}}(M') \simeq F_{\mathcal{E}}(M')$ for all $M' \in \rep(\Lambda_{\mathcal{E}},\dd')$;
%\item $\GL(\dd) \cdot \Ima f_{\mathcal{E}}=\filt_{\mathcal{E}}(\dd)$;
\item $f_{\mathcal{E}}(g' \cdot M')=\varphi(g')\cdot f_{\mathcal{E}}(M')$ for all $M' \in \rep(\Lambda_{\mathcal{E}},\dd')$ and $g' \in \GL(\dd')$.
\end{enumerate}
\end{proposition}

\begin{remark} Note that this proposition already appears in the context of quivers with no relations in \cite{RieSch1}. Furthermore, the proof of Proposition 2.3 in \cite{RieSch1} works for arbitrary finite-dimensional algebras, as well. Nonetheless, for completeness, we include below an explicit proof.  
\end{remark}

\begin{proof} In what follows, we denote by $e'_i$ the primitive idempotent in $\Lambda_{\mathcal{E}}$ corresponding to vertex $i \in \{1,2,\ldots, l\}$. Furthermore, if $i$ and $j$ are two vertices of $Q_{\mathcal E}$, $p'$ is a linear combination of paths in $kQ_{\mathcal E}$ from vertex $i$ to vertex $j$, and $M' \in \rep(\Lambda_{\mathcal E})$, we define $M'(p')$ to be the corresponding linear combination of products of matrices. We also denote the residue class of $p'$ in $\Lambda_{\mathcal E}$ by $\overline{p'}$.

According to Theorem 2 in  \cite{Wat}, we can assume that $F=P\otimes_{\Lambda_{\mathcal E}} \underline{\phantom{X}}$ where $P$ is a finite-dimensional $\Lambda-\Lambda_{\mathcal{E}}$-bimodule which is projective as a right $\Lambda_{\mathcal{E}}$-module. In fact, we can write $P=\bigoplus_{v\in Q_0}e_vP$ where $e_vP=\bigoplus_{i=1}^l (e'_i\Lambda_{\mathcal{E}})^{\dd_{E_i}(v)}, \forall v \in Q_0$.

Since $P$ is a $\Lambda-\Lambda_{\mathcal E}$-bimodule, we have that for each arrow $a\in Q_1$, $P(a) \in Hom_{\Lambda_{\mathcal{E}}}(e_{ta}P,e_{ha}P)$ which, after fixing bases, can be viewed as an $l \times l$ block-matrix whose $(i,j)$-block entry is a matrix of size $\dd_{E_i}(ha)\times \dd_{E_j}(ta)$ whose entries are of the form $\overline{p'}\in \Lambda_{\mathcal{E}}$ with $p'$ a linear combination of paths from vertex $j$ to vertex $i$ in $kQ_{\mathcal{E}}$.

Now, given a representation $M' \in \rep(\Lambda_{\mathcal E},\dd')$ and an arrow $a \in Q_1$, we define $f_{\mathcal E}(M')(a)$ to be the $l\times l$ block-matrix obtained from $P(a)$ by replacing each entry $\overline{p'} \in e'_{i}\Lambda_{\mathcal E} e'_j$ by the $\dd'(i) \times \dd'(j)$ matrix $M'(p')$. As for the morphism  $\varphi$, we simply take the natural diagonal embedding of $\GL(\dd')$ into $\GL(\dd)$. The proof now follows.    

\end{proof}

For the convenience of the reader, we now recall some fundamental facts from birational invariant theory that will be used in the proof of Theorem \ref{rational-inv-quiverel-thm}. Let $G$ be a linear algebraic group acting regularly on an irreducible variety $X$. The field $k(X)^G$ of $G$-invariant rational functions on $X$ is always finitely generated over $k$ since it is a subfield of $k(X)$ which is finitely generated over $k$. A rational quotient of $X$ by (the action of) $G$ is an irreducible variety $Y$ such that $k(Y)=k(X)^G$ together with the dominant rational map $\pi:X \dashrightarrow Y$ induced by the inclusion $k(X)^G \subset k(X)$. Of course, $Y$ is uniquely defined up to birational isomorphism.

Now, a theorem of Rosenlicht \cite{Ros} tells us that there is a $G$-invariant open and dense subset $X_0$ of $X$ such that the restriction of $\pi$ to $X_0$ is a dominant regular  morphism and $\pi^{-1}(\pi(x))=Gx$ for all $x \in X_0$. Furthermore, one can show that a rational quotient $\pi:X \dashrightarrow Y$ satisfies the following universal property (see \cite[Section 2.4]{PopVin}, \cite[Remark 2.5]{Rei1}): Let $\rho:X\dashrightarrow Y'$ be a rational map such that $\rho^{-1}(\rho(x))=Gx$ for $x \in X$ in general position. Then there exists a rational map
$$
\overline{\rho}:Y \dashrightarrow Y'
$$
such that $\rho=\overline{\rho} \circ \pi$. If in addition $\rho$ is dominant then $\overline{\rho}$ becomes a birational isomorphism. One usually writes $X/G$ in place of $Y$ and call it \emph{the} rational quotient of $X$ by $G$.

%\begin{remark} Let $C$ be an irreducible component of $\rep(\Lambda,\dd)$. If $\theta$ is any integral weight for which $C$ is $\theta$-stable then %the moduli space $\M(\Lambda, C)^{ss}_{\theta}$ is (birationally isomorphic to) the rational quotient $C/\PGL(\dd)$.
%\end{remark}

We also need some facts about homogeneous fiber spaces. Let $\varphi: H \to G$ be a homomorphism of algebraic groups and let $Z$ be an $H$-variety. Consider the action of $G\times H$ on $G\times Z$ defined by $(g,h)\cdot (g',z)=(gg' \varphi(h)^{-1},hz)$ and denote by $G\ast_{H}Z$ the rational quotient of $G \times Z$ by the above action of $\{1\}\times H$. We call $X:=G\ast_{H}Z$ a \emph{homogeneous fiber space}. Note that $G$ has a naturally defined rational action on $X$ which, in general, is not regular. However, it is always possible (see \cite[Definition 2.12]{Rei1}) to choose a model $Y$ for $X$ such that the $G$-action on $G \times Z$ descends to a regular action on $Y$, and thus the rational quotient map $\pi: G \times Z \dashrightarrow Y$ is $G$-equivariant.

%\begin{remark} Let $C$ be an irreducible $G$-variety and let $f:Z \dashrightarrow C$ be an $H$-equivariant rational map. Define $\mu: G\times Z\to C$ to be the map that  sends $(g,z)$ to $gf(z)$. Then, $\mu$ descends to a $G$-equivariant rational map from ... to $C$.
%\end{remark}

Let us record the following very useful result:

\begin{lemma}\cite[Lemma 6.1]{ReiVon1}\label{fiber-space} Keep the same notation as above. Then, $k(X)^G$ and $k(Z)^H$ are isomorphic as $k$-algebras. \end{lemma}

Now, we are ready to prove Theorem \ref{rational-inv-quiverel-thm}.

\begin{proof}[Proof of Theorem \ref{rational-inv-quiverel-thm}] Denote by $T'_1$ the $1$-dimensional torus in $\GL(\dd')$ and by $T_1$ the $1$-dimensional torus in $\GL(\dd)$. Since $\filt_{\mathcal{E}}(\dd)$ contains a Schur $\Lambda$-module by assumption, we immediately deduce that $\varphi (T'_1) \leq T_1$. Hence, we have a well-defined action of $\PGL(\dd')$ on $\PGL(\dd)$ induced by the action $g'\cdot g=g\cdot \varphi(g')^{-1}$. %Furthermore, note that $\PGL(\dd')$ acts on $\rep(\Lambda_{\mathcal{E}},\dd')$ generically free, and hence, the same is true for the action of $\PGL(\dd')$ on $\PGL(\dd)\times \rep(\Lambda_{\mathcal{E}},\dd')$.

Let us now consider the $\PGL(\dd')$-invariant morphism $\mu: \PGL(\dd) \times \rep(\Lambda_{\mathcal{E}},\dd') \to C$ induced by $f_{\mathcal{E}}$. By Theorem \ref{A-infinity-functor-open-thm}{(2)} and Proposition \ref{morphisms-prop}, we know that $C=\overline{\filt_{\mathcal{E}}(\dd)}=\overline{GL(\dd)\cdot \Ima f_{\mathcal{E}}}$, i.e., $\mu$ is a dominant morphism.

Next, let $(\overline{g},M) \in \PGL(\dd)\times \rep(\Lambda_{\mathcal{E}},\dd')$ be a generic point with $M$ a Schur $\Lambda_{\mathcal{E}}$-module. (Here, $\overline{g}$ denotes the image of $g \in \GL(\dd)$ in $\PGL(\dd)$.) Then we claim that $\mu^{-1}(\mu(\overline{g},M))=\PGL(\dd') \cdot (\overline{g},M)$. Indeed, let $(h, M') \in \GL(\dd)\times \rep(\Lambda_{\mathcal E},\dd')$ be so that $hf_{\mathcal E}(M')=gf_{\mathcal E}(M)$. In particular, $M \simeq M'$ and so there exists an $g' \in \GL(\dd')$ such that $M'=g'M$. Since $f_{\mathcal{E}}(M)$ is a Schur module, we get that $\overline{h}=\overline{g'\cdot g}$, and so $\overline{g'}(\overline{g},M)=(\overline{h}, M')$. The claim now follows. 

From the uniqueness of rational quotients, we know that $\mu$ is actually the rational quotient map for the $\PGL(\dd')$-action on $\PGL(\dd)\times \rep(\Lambda_{\mathcal{E}},\dd')$, i.e., $\mu$ descends to a birational isomorphism $$\overline{\mu}: \PGL(\dd)\ast_{\PGL(\dd')} \rep(\Lambda_{\mathcal{E}},\dd') \dashrightarrow C$$
such that $\mu=\overline{\mu}\circ \pi$, where $\pi: \PGL(\dd) \times \rep(\Lambda_{\mathcal{E}},\dd') \dashrightarrow \PGL(\dd)\ast_{\PGL(\dd')} \rep(\Lambda_{\mathcal{E}},\dd')$ is the rational quotient map. Note that as $\pi$ is $\PGL(\dd)$-equivariant, so is $\overline{\mu}$.
%Note that $\overline{\mu}$ is $\PGL(\dd)$-equivariant since the $\PGL(\dd)$-action on $\PGL(\dd)\times \rep(\Lambda_{\mathcal{E}},\dd')$ commutes with that of $\PGL(\dd')$.
The proof now follows from Lemma \ref{fiber-space}.
\end{proof}

As an immediate consequence of Theorem \ref{rational-inv-quiverel-thm}, we have:

\begin{corollary} \label{rational-inv-Euclidean-coro} If $Q$ is a Euclidean quiver then $k(\rep(Q,\dd))^{\GL(\dd)} \simeq k$ or $k(t)$ for each Schur root $\dd$ of $Q$.
\end{corollary}

\begin{proof} If $\dd$ is a real Schur root then $k(\rep(Q,\dd))^{\GL(\dd)} \simeq k$. Next, we denote by $\delta_Q$, the unique isotropic Schur root of $Q$. Choose a vertex $i_0 \in Q_0$ such that $Q \setminus \{i_0\}$ is a Dynkin quiver. Without loss of generality, let us assume that $i_0$ is a source. In this case, we can choose two exceptional representations $E_1$ and $E_2$ of $Q$ with $\ddim E_1=\delta_Q-e_{i_0}$ and $\ddim E_2=e_{i_0}$. Then,  $\mathcal{E}:=(E_1,E_2)$ is an orthogonal exceptional sequence with $$\dim_k \Ext^1_{kQ} (E_2, E_1)=2.$$ Hence, $Q(\mathcal{E})$ is the Kronecker quiver
$$K_2:\hspace{15pt}
\xy     (0,0)*{\cdot}="a";
        (10,0)*{\cdot}="b";
        {\ar@2{<-} "a";"b" };
\endxy
$$
But, for $K_2$ and dimension vector $(1,1)$, the corresponding field of rational invariants is clearly $k(t)$. The proof now follows from Theorem \ref{rational-inv-quiverel-thm}.
\end{proof}

\begin{remark} Note that the sequence $\mathcal{E}$ above corresponds to a facet of the cone $\Eff(Q,\delta)$ of effective weights associated to $(Q,\delta)$. In fact, the facets of $C(Q,\delta)$ give rise to all orthogonal exceptional sequences $\mathcal{E}$ of $Q$ of length two for which $\filt_{\mathcal{E}}(\delta)\neq \emptyset$ (see \cite[Theorem 5.1]{DW2}). We will come back to this approach in Section \ref{exceptional-effective-wts-sec}.
\end{remark}

\section{Canonical algebras}\label{canonical-sec} Let $\underline{m}=(m_1,\dots,m_n)$, $n \geq 3$, be a sequence of positive integers greater than one, and let $\underline{\lambda}=(\lambda_3, \dots, \lambda_n)$ be a sequence of pairwise distinct non-zero scalars in $k$ with $\lambda_3=1$. The canonical algebra $\Lambda=\Lambda(\underline{m},\underline{\lambda})$ is, by definition, the bound quiver algebra of the bound quiver $(\Delta(\underline{m}), R(\underline{m},\underline{\lambda}))$ where $\Delta(\underline{m})$ is the quiver:

$$
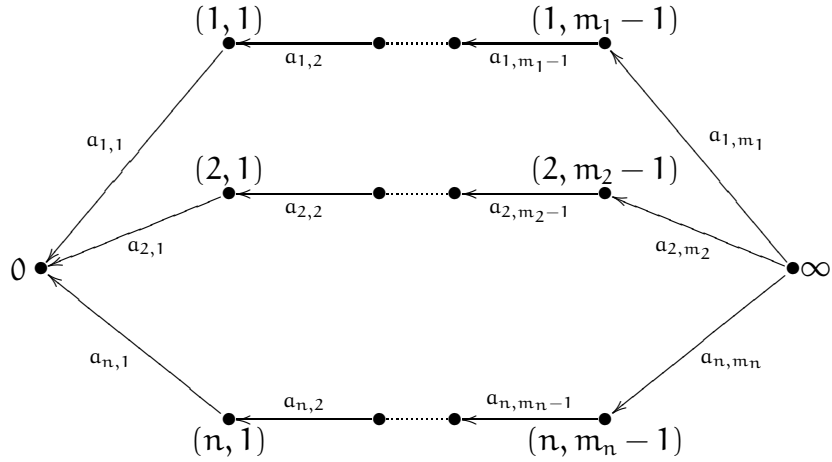
\xy (0, 0)*{\bullet}="a";(3,0)*{\infty};
        (-25, 30)*{\bullet}="a1";(-25,33)*{(1,m_1-1)};
        (-45,30)*{\bullet}="a2";
        (-55,30)*{\bullet}="a3";
        (-75,30)*{\bullet}="a4";(-75,33)*{(1,1)};
        (-100,0)*{\bullet}="b";(-103,0)*{0};
        (-25, 10)*{\bullet}="b1";(-25,13)*{(2,m_2-1)};
        (-45,10)*{\bullet}="b2";
        (-55,10)*{\bullet}="b3";
        (-75,10)*{\bullet}="b4";(-75,13)*{(2,1)};
        (-25,-20)*{\bullet}="c1";(-25,-23)*{(n,m_n-1)};
        (-45,-20)*{\bullet}="c2";
        (-55,-20)*{\bullet}="c3";
        (-75,-20)*{\bullet}="c4";(-75,-23)*{(n,1)};
        {\ar_{a_{1,m_1}} "a";"a1"};
        {\ar^{a_{1,m_1-1}} "a1";"a2"};
        {\ar@{.} "a2";"a3"};
        {\ar^{a_{1,2}} "a3";"a4"};
        {\ar^{a_{2,m_2}} "a";"b1"};
        {\ar^{a_{2,m_2-1}} "b1";"b2"};
        {\ar@{.} "b2";"b3"};
        {\ar^{a_{2,2}} "b3";"b4"};
        {\ar^{a_{n,m_n}} "a";"c1"};
        {\ar_{a_{n,m_n-1}} "c1";"c2"};
        {\ar@{.} "c2";"c3"};
        {\ar_{a_{n,2}} "c3";"c4"};
        {\ar_{a_{1,1}} "a4";"b"};
        {\ar^{a_{2,1}} "b4";"b"};
        {\ar^{a_{n,1}} "c4";"b"};
    \endxy
$$

and $R(\underline{m},\underline{\lambda})$ consists of the following relations:
$$
a_{1,1}a_{1,2}\dots a_{1,m_1}+\lambda_i a_{2,1}a_{2,2}\dots a_{2,m_2}-a_{i,1}a_{i,2}\dots a_{i,m_i}, 3 \leq i \leq n.
$$

The virtual genus of $\Lambda$, denoted by $g_{\Lambda}$, is
$$g_{\Lambda}=1+{1 \over 2}m(n-2-{1 \over m_1}-\dots-{1 \over m_n}),$$
where $m=lcm\{m_1,\ldots, m_n\}$. Note that $g_{\Lambda}$ is always non-negative. Moreover, the virtual genus $g_{\Lambda}$ controls the representation type of $\Lambda$ in the following way (see \cite{R3} or \cite[Section 7]{BobRieSko}).
\begin{enumerate}
\renewcommand{\theenumi}{\arabic{enumi}}
\item If $g_{\Lambda}=0$ then $n=3$ and $\underline{m}$ is one of the following triples $(l-2,2,2)$ with $l \geq 4$, $(3,3,2)$, $(4,3,2)$ or $(5,3,2)$, i.e., $\Delta \setminus \{\infty\}$ is a Dynkin quiver of type $\mathbb{D}$ or $\mathbb{E}$. In this case, $\Lambda$ is a concealed algebra of extended Dynkin type $\widetilde{\mathbf{D}}_l$, $\widetilde{\mathbf{E}}_6$, $\widetilde{\mathbf{E}}_7$, or $\widetilde{\mathbf{E}}_8$ (see for example the D. Happel-D. Vossieck's list in \cite{HapVos}). 
\item If $g_{\Lambda}=1$ (or equivalently, $0<g_{\Lambda} \leq 1$) then $\underline{m}$ is one of the following four tuples $(2,2,2,2)$, $(3,3,3)$, $(4,4,2)$, and $(6,3,2)$, i.e., $\Delta \setminus \{\infty\}$ is an extended Dynkin diagram of type $\widetilde{\mathbb{D}}$ or $\widetilde{\mathbb{E}}$. In this case, we call $\Lambda$ a \emph{tubular canonical algebra}. The classification of the indecomposable modules over a tubular canonical algebra turns out to be closely related to Atiyah's \cite{Ati} classification of indecomposable bundles over an elliptic curve.

\item $\Lambda$ is wild if and only if $g_{\Lambda}>1$.
\end{enumerate}

In what follows, we briefly recall some of the key features of canonical algebras (see for example \cite{R3} or  \cite{LenMel}). First of all, $\Lambda$ has global dimension two. In particular, the Tits form $q_{\Lambda}$ coincides with $\chi_{\Lambda}$. The \emph{rank} and \emph{degree} of a dimension vector $\dd$ of $\Lambda$, denoted by $\rk_{\Lambda}(\dd)$ and $\ddeg_{\Lambda}(\dd)$ respectively, are
$$
\rk_{\Lambda}(\dd)=\dd(0)-\dd(\infty)
$$
and
$$
\ddeg_{\Lambda}(\dd)=\sum_{i=1}^{n}{m \over m_i} \left( \sum_{j=1}^{m_i-1} \dd(i,j) \right)-\left((n-1)m-\sum_{i=1}^{n}{m \over m_i}  \right) \dd(\infty).
$$
We denote by $\hh$ the dimension vector of $\Lambda$ that takes value $1$ at all vertices of $\Delta_0$. It turns out that $\rk_{\Lambda}(\dd)=\langle \dd,\hh \rangle_{\Lambda}=-\langle \hh,\dd \rangle_{\Lambda}$ for any dimension vector $\dd$ of $\Lambda$.

Let $\mathcal P$ ($\mathcal R, \mathcal Q$, respectively) be the full subcategory of $\module(\Lambda)$ consisting of all $\Lambda$-modules that are direct sums of indecomposable $\Lambda$-modules $X$ such that $\rk_{\Lambda}(\ddim X)>0$ ($\rk_{\Lambda}(\ddim X)=0, \rk_{\Lambda}(\ddim X)<0$, respectively). The following properties hold true.
\begin{enumerate}
\renewcommand{\theenumi}{\roman{enumi}}
\item $\module(\Lambda)=\mathcal P \bigvee \mathcal R \bigvee \mathcal Q$.
\item $\Hom_{\Lambda}(N,M)=\Ext^1_{\Lambda}(M,N)=0$ if either $N \in \mathcal R \bigvee \mathcal Q, M \in \mathcal P$ or  $N \in \mathcal Q, M \in \mathcal P \bigvee \mathcal R$.
\item $pd_{\Lambda} M\leq 1$ for all $M \in \mathcal P \bigvee \mathcal R$ and $id_{\Lambda} N \leq 1$ for all $N \in \mathcal R \bigvee \mathcal Q$.
%\item Denote by $\mathbf P$ and $\mathbf Q$ the sets of the dimension vectors of the $\Lambda$-modules in $\mathcal P$ and $\mathcal Q$, respectively. A dimension vector $\dd$ of $\Lambda$ belongs to $\mathbf P$ ($\mathbf Q$, respectively) if and only if $\dd=\mathbf 0$ or $\dd(0)>\dd(\infty)\geq 0$ ($0 \leq \dd(0)<\dd(\infty)$, respectively) and $\dd(i,j-1)\geq \dd(i,j)$ ($\dd(i,j-1)\leq \dd(i,j)$, respectively) for all $1 \leq i \leq n$ and $1 \leq j \leq m_i$. The convention here is that $\dd(i,0)=\dd(0)$ and $\dd(i,m_i)=\dd(\infty)$ for all $1 \leq i \leq n$.
\end{enumerate}

%It follows from (iv) that if $\dd$ is the dimension vector of a $\Lambda$-module from $\mathcal P$ (or $\mathcal Q$) and if $\dd$ is not sincere then $\supp(\dd):=\{v \in \Delta_0\mid \dd(v)\neq 0\}$ is included in $\Delta_0\setminus \{\infty\}$ (or  $\Delta_0\setminus \{0\}$).

We end this subsection with the Riemann-Roch formula for canonical algebras due to Geigle-Lenzing \cite{GeiLen} (see also \cite{LenMel} and \cite{Len}). Denote by $\Phi_{\Lambda}$ the Coxeter matrix of $\Lambda$. The Riemann-Roch formula tells us that for any two dimension vectors $\dd$ and $\ee$ of $\Lambda$:

\begin{equation} \label{RR-formula}
\sum_{i=0}^{m-1}\langle \Phi_{\Lambda}^i \dd,\ee \rangle_{\Lambda}=m(1-g_{\Lambda})\rk_{\Lambda}(\dd)\rk_{\Lambda}(\ee)+\det \left(
\begin{matrix}
\rk_{\Lambda}(\dd)&\rk_{\Lambda}(\ee)\\
\ddeg_{\Lambda}(\dd)&\ddeg_{\Lambda}(\ee)
\end{matrix}
\right).
\end{equation}

\subsection{Irreducible components for tame canonical algebras} In this section, we review some results of Bobi{\'n}ski and Skowro{\'n}ski \cite{BS1}, and of Geiss and Schr{\"o}er \cite{GeiSch} on the indecomposable irreducible components for a tame canonical algebra $\Lambda$.

\begin{theorem}\label{BS-GS-irr-comp-thm} Let $\Lambda$ be a tame canonical algebra and let $\dd$ be a generic root of $\Lambda$. Then $\dd$ is an indivisible Schur root of $\Lambda$, $q_{\Lambda}(\dd)\in \{0,1\}$, and $\rep(\Lambda, \dd)$ has a unique indecomposable irreducible component. More precisely, if $q_{\Lambda}(\dd)=1$ then there exists a unique $\dd$-dimensional exceptional $\Lambda$-module $M$ and $\overline{\GL(\dd)M}$ is the unique indecomposable irreducible component of $\rep(\Lambda,\dd)$. If $q_{\Lambda}(\dd)=0$ then $\rep(\Lambda,\dd)$ is irreducible.
\end{theorem}

\begin{remark} Note that when $\Lambda$ is tame concealed the only Schur root $\dd$ for which $q_{\Lambda}(\dd)=0$ is precisely $\hh$ (see for example \cite{HapVos}). Moreover, $\hh$ generates the radical of $\chi_{\Lambda}$ in this case.
\end{remark}

The dimension vector of an indecomposable $\Lambda$-module is called a \emph{root} of $\Lambda$. A root $\dd$ of $\Lambda$ is  said to be \emph{real} if $q_{\Lambda}(\dd)=1$. We call the root $\dd$ of $\Lambda$ \emph{isotropic} $q_{\Lambda}(\dd)=0$. 

Using the Riemann-Roch formula $(\ref{RR-formula})$ and the fact that the Schur roots of tame canonical algebras are indivisible, one can show:

\begin{lemma}\label{iso-RR-lemma} Let $\Lambda$ be a tubular canonical algebra. If $\dd$ is an isotropic Schur root of $\Lambda$ then $\Phi_{\Lambda} \dd=\dd$. In particular,
$$
\langle \dd,\ee \rangle_{\Lambda}={1 \over m} \det \left(
\begin{matrix}
\rk_{\Lambda}(\dd)&\rk_{\Lambda}(\ee)\\
\ddeg_{\Lambda}(\dd)&\ddeg_{\Lambda}(\ee)
\end{matrix}
\right),$$
for all dimension vectors $\ee$ of $\Lambda$.
\end{lemma}

\begin{remark} Note that the condition $\Phi_{\Lambda}\dd=\dd$ is equivalent to $\langle \dd,\ee \rangle_{\Lambda}+\langle \ee,\dd \rangle_{\Lambda} =0$ for all $\ee \in \ZZ^{\Delta_0}$, i.e., $\dd$ is in the radical of $\chi_{\Lambda}$.
\end{remark}

\begin{proof} Since $\Lambda$ is tubular, we know that $\Phi_{\Lambda}^m \dd'=\dd'$ for any dimension vectors $\dd'$ of $\Lambda$ (see for example \cite{Len}). Given a dimension vector $\dd'$ of $\Lambda$, set $r(\dd')=\min \{ i\geq 1 \mid \Phi_{\Lambda}^i\dd'=\dd' \}$ and $l(\dd')=g.c.d\{\left( \sum_{i=0}^{r(\dd')-1}\Phi_{\Lambda}^i\dd'  \right)(v)\mid v \in \Delta_0  \}$. Then, $\iso(\dd'):={r(\dd') \over m \cdot l(\dd')}\sum_{j=0}^{m-1}\Phi_{\Lambda}^j\dd'$ is an indivisible isotropic Schur root of $\Lambda$ such that $\Phi_{\Lambda} \iso(\dd')=\iso(\dd')$.

Using the Riemann-Roch formula $(\ref{RR-formula})$, we obtain that $\langle \iso(\dd),\dd \rangle_{\Lambda} =0$. It now follows from the general theory of tubular algebras that the two isotropic roots $\iso(\dd)$ and $\dd$ are multiple of each other, and so they must be equal as they are both indivisible. The proof now follows.
\end{proof}

\subsection{Exceptional sequences from cones of effective weights}\label{exceptional-effective-wts-sec} In what follows we provide a systematic approach to finding ``convenient'' orthogonal exceptional sequences of modules. This approach is based on the notion of $\theta$-stable decomposition of dimension vectors in irreducible components of module varieties. From this point on until Lemma \ref{lemma-dim-Eff} below,  $\Lambda=kQ/\langle R \rangle$ is the bound quiver algebra of an arbitrary bound quiver $(Q,R)$, $\dd$ is a dimension vector of $\Lambda$, and $\theta \in \RR^{Q_0}$. Recall that a module $M \in \rep(\Lambda,\dd)$ is said to be \emph{$\theta$-semi-stable} if $\theta(\ddim M)=0$ and $\theta(\ddim M') \leq 0$ for all submodules $M' \subseteq M$. We say that $M$ is \emph{$\theta$-stable} if  $\theta(\ddim M)=0$ and $\theta(\ddim M') < 0$ for all proper submodules $\{0\} \subset M' \subset M$. Denote by $\rep(\Lambda)^{s(s)}_{\theta}$ the full subcategory of $\rep(\Lambda)$ consisting of all $\theta$-(semi-)stable $\Lambda$-modules. Then, $\rep(\Lambda)^{ss}_{\theta}$ is an abelian subcategory of $\rep(\Lambda)$ which is closed under extensions, and whose simple objects are precisely the $\theta$-stable $\Lambda$-modules. Moreover, $\rep(\Lambda)^{ss}_{\theta}$ is Artinian and Noetherian, and hence, every $\theta$-semi-stable finite-dimensional $\Lambda$-module has a Jordan-H{\"o}lder filtration in $\rep(Q)^{ss}_{\theta}$.

Now, let $C$ be an irreducible component of $\rep(\Lambda,\dd)$. For a real weight $\theta \in \RR^{Q_0}$, we define $C^{s(s)}_{\theta}=\{M \in C \mid M \text{~is~} \theta\text{-(semi-)stable} \}$. Next, the cone of effective weights of $C$ is, by definition, the set
$$
\Eff(C)=\{\theta \in \RR^{Q_0}\mid C^{ss}_{\theta}\neq \emptyset \}.
$$

We know that there are only finitely many GIT-classes in the cone $\Eff(Q,\dd)$ of effective weights associated to $(Q,\dd)$ (see for example \cite{ArHa2} or \cite{CC5}). Among a set of representatives for these GIT-classes, we denote by $\theta_1,\ldots, \theta_l$ the integral weights for which the corresponding semi-stable loci in $C$ are non-empty. Note that for any $\theta \in \Eff(C)$, $C^{ss}_{\theta}$ is open in $C$. Moreover, for any representation $M \in \bigcap_{i=1}^l C^{ss}_{\theta_i}$, we have
$$
\Eff(C)=\{\theta \in \RR^{Q_0} \mid \theta(\dd)=0 \text{~and~} \theta(\dd_{M'})\leq 0, \forall M'\subseteq M\},
$$
and so $\Eff(C)$ is a rational convex polyhedral cone.

Let $\theta$ be a lattice point in $\Eff(C)$. For each sequence $\mathcal D=(\dd_1,\ldots, \dd_t)$ of dimension vectors of $\theta$-stable $\Lambda$-modules, consider the subset $C_{\mathcal{D}}$ of $C$ consisting of all $\Lambda$-modules $M \in C^{ss}_{\theta}$ that have a Jordan-H{\"o}lder filtration $\{0\}=M_0 \subset M_1 \subset \ldots \subset M_t=M$ in $\rep(\Lambda)^{ss}_{\theta}$ such that the sequence $(\ddim M_0,\ddim M_1/M_0, \ldots, \ddim M/M_{t-1})$ is the same as $\mathcal D$ up to permutation. It is not difficult to see that $C_{\mathcal D}$ is a constructible subset of $C$ (see for example \cite{C-BS}). Since $C^{ss}_{\theta}$ is irreducible, we deduce that there exists a unique, up to permutation, such sequence $\mathcal D=(\dd_1,\ldots,\dd_t)$ for which $C_{\mathcal D}$ contains an open and dense subset of $C$. We write
$$
\dd=\dd_1\pp \ldots \pp \dd_t,
$$
and call it the \emph{$\theta$-stable decomposition} of $\dd$ in $C$.

In what follows, for a given $\dd \in \RR^{Q_0}$, we denote by $\mathbf H(\dd)$ the hyperplane in $\RR^{Q_0}$ orthogonal to $\dd$, i.e., $\mathbf H(\dd)=\{\theta \in \RR^{Q_0} \mid \theta(\dd)=0\}$.

\begin{lemma}\label{face-Eff-lemma} Let $\mathcal F$ be a face of $\Eff(C)$ of positive dimension and let $\theta_0 \in \rel \Eff(C) \cap \ZZ^{Q_0}$. If
$\dd=m_1\cdot \dd_1 \pp \ldots m_t \cdot \dd_t$ is the $\theta_0$-stable decomposition of $\dd$ in $C$ with $\dd_i\neq \dd_j, \forall 1 \leq i<j \leq t$, then
$$
\mathcal F=\Eff(C) \cap \bigcap_{i=1}^t \mathbf H(\dd_i).
$$
\end{lemma}

\begin{proof} It follows from the discussion above that we can always choose a module $M \in C^{ss}_{\theta_0}$ such that
\begin{itemize}
\item $\Eff(C)=\{\theta \in \RR^{Q_0} \mid \theta(\dd)=0 \text{~and~} \theta(\dd_{M'})\leq 0, \forall M' \subseteq M\}$, and

\item  $M$ has a Jordan-H{\"o}lder filtration $\{0\}=M_0 \subset M_1 \subset \ldots \subset M_N=M$ in $\rep(\Lambda)^{ss}_{\theta_0}$ such that the sequence $(\ddim M_0,\ddim M_1/M_0, \ldots, \ddim M/M_{N-1})$ is the same as $(\dd_1^{m_1},\ldots, \dd_t^{m_t})$ up to permutation. (Here, $N=m_1+\ldots+m_t$.)
\end{itemize}

In particular, we get that
$$
\mathcal F=\Eff(C)\cap \bigcap \{\theta \in \RR^{Q_0} \mid \theta(\ddim M')=0\},
$$
where the union is over all submodules $M'$ of $M$ for which $\theta_0(\ddim M')=0$.

Now, we clearly have that $\theta_0(\ddim M_i)=0, \forall 1 \leq i \leq t$, and so $\mathcal F \subseteq \Eff(C) \cap \bigcap_{i=1}^t \mathbf H(\dd_i)$. To show the other inclusion, first note that if $M' \subseteq M$ is a submodule such that $\theta_0(\ddim M')=0$ then $M'$ is $\theta_0$-semi-stable, and using the uniqueness of the factors of a Jordan-H{\"o}lder filtration in $\rep(\Lambda)^{ss}_{\theta_0}$, we deduce that $\ddim M'$ is a linear combination of some of $\dd_1,\ldots, \dd_t$. The other inclusion now follows.
\end{proof}

We also have the following useful lemma:

\begin{lemma}\label{lemma-dim-Eff} Let $\Lambda$ be a tame canonical algebra and let $\dd$ be an isotropic Schur root of $\Lambda$. Then, $\rep(\Lambda,\dd)^{s}_{\theta_{\dd}} \neq \emptyset$ where $\theta_{\dd}$ denotes the weight $\langle \dd, \cdot \rangle_{\Lambda}$.
\end{lemma}

\begin{proof} From the general theory of tame concealed algebras and of tubular algebras (see for example \cite[Section 2]{BobSko}), we know that $\dd$ is the dimension vector of an indecomposable $\Lambda$-module lying at the mouth of a homogeneous tube which is part of a family, call it $\mathcal T$, of pairwise orthogonal tubes. Specifically, $\mathcal T$is the full subcategory of $\module(\Lambda)$ consisting of all $\Lambda$-modules that are direct sums of indecomposable $\Lambda$-modules $X$ such that $\theta_{\dd}(\ddim X)=0$. Moreover, let $\widetilde{\mathcal P}$ ($\widetilde{\mathcal Q}$, respectively) be the full subcategory of $\module(\Lambda)$ consisting of all $\Lambda$-modules that are direct sums of indecomposable $\Lambda$-modules $X$ such that $\theta_{\dd}(\ddim X)<0$ ($\theta_{\dd}(\ddim X)>0$, respectively). Then, we have:
\begin{enumerate}
\renewcommand{\theenumi}{\arabic{enumi}}
\item $\module(\Lambda)=\widetilde{\mathcal P} \bigvee  \mathcal T \bigvee  \widetilde{\mathcal Q}$;
\item $\Hom_{\Lambda}(X,Y)=0$ if either $X \in \widetilde{\mathcal Q}, Y \in \mathcal T$ or $X \in \mathcal T, Y \in \widetilde{\mathcal P}$.
\end{enumerate}

It is now clear that $\rep(\Lambda,\dd)^{ss}_{\theta_{\dd}}\neq \emptyset$ since any $\dd$-dimensional $\Lambda$-module from $\mathcal T$ is $\theta_{\dd}$-semi-stable. Let $M\in \rep(\Lambda,\dd)$ be an indecomposable module that lies in a homogeneous tube of $\mathcal T$. We are going to show that $M$ is $\theta_{\dd}$-stable. Assume to the contrary that $M$ is not $\theta_{\dd}$-stable and consider a Jordan-H{\"o}lder filtration of $M$ in $\rep(\Lambda)^{ss}_{\theta_{\dd}}$. This way, we can see that $M$ must have a proper $\theta_{\dd}$-stable submodule $M'$. Then, $M'$ must belong to the homogeneous tube of $M$, and from this we deduce that $\ddim M'$ is an integer multiple of $\dd$. But this is a contradiction.
\end{proof}

Now, we are ready to prove:

%We end this section with the following lemma which is needed in order to apply the reduction Theorem \ref{rational-inv-quiverel-thm} to canonical algebras.

\begin{proposition} \label{exceptional-seq-tubular-prop} Let $\Lambda$ be a tame canonical algebra and let $\dd$ be an isotropic Schur root of $\Lambda$. Then, there exists an orthogonal exceptional sequence $\mathcal{E}=(E_1,E_2)$ of $\Lambda$-modules such that $\filt_{\mathcal{E}}(\dd)$ contains a Schur $\Lambda$-module and $\Lambda_{\mathcal E}$ is the path algebra of the Kronecker quiver $K_2$.
\end{proposition}

\begin{proof} We know that $\rep(\Lambda,\dd)$ is irreducible by Theorem \ref{BS-GS-irr-comp-thm} and let us denote by $\Eff(\Lambda,\dd)$ the cone of effective weights of $\rep(\Lambda,\dd)$. It follows from Lemma \ref{lemma-dim-Eff} that $\dim \Eff(\Lambda,\dd)=|\Delta_0|-1$. Next, choose a facet $\mathcal{F}$ of the cone $\Eff(\Lambda,\dd)$ and a weight $\theta_0\in \rel \mathcal F \cap \ZZ^{\Delta_0}$. Now, consider the $\theta_0$-stable decomposition of $\dd$ in $\rep(\Lambda,\dd)$:
$$\dd=m_1\cdot \dd_1 \pp \ldots \pp m_t \cdot \dd_t,$$ with $m_1,\ldots, m_t$ positive integers and $\dd_i\neq \dd_j, \forall 1 \leq i \neq j \leq t$. Note that $\dd_1, \ldots, \dd_t$ are indivisible Schur roots by Theorem \ref{BS-GS-irr-comp-thm}.

For each $1 \leq i \leq t$, let $E_i$ be a $\dd_i$-dimensional $\theta_0$-stable module that arises as a factor of a Jordan-H{\"o}lder filtration of a generic module $M$ in $\rep(\Lambda,\dd)$. Note that we can choose $M$ to be $\theta_{\dd}$-stable by Lemma \ref{lemma-dim-Eff}. Furthermore, we have that $\Hom_{\Lambda}(E_i,E_j)=0, \forall 1 \leq i\neq j\leq t$ since $E_1, \ldots, E_t$ are pairwise non-isomorphic ($\theta_0$-)stable modules.

\underline{Claim:} $\mathcal F=\Eff(\Lambda,\dd)\cap \mathbf{H}(\dd_1)\cap \mathbf{H}(\dd_2)$ and $\dd=n_1\dd_1+n_2\dd_2$ for unique numbers $n_1$ and $n_2$.

\begin{proof}[Proof of Claim:] Note that $\mathcal{F}$ has dimension $|\Delta_0|-2$, and so $t \geq 2$. Moreover, the dimension of the subspace of $\RR^{\Delta_0}$ orthogonal to the subspace spanned by $\{\dd, \dd_1,\dd_2\}$ is at least $\Delta_0-2$ since it contains $\mathcal F$. In particular, the set $\{\dd,\dd_1,\dd_2\}$ is linearly dependent. Since $\dd_1$ and $\dd_2$ are distinct indivisible vectors, we deduce that $\dd=n_1\dd_1+n_2\dd_2$ for unique numbers $n_1$ and $n_2$.

When $t=2$, the proof follows from Lemma \ref{face-Eff-lemma}. Now, let us assume that $t \geq 3$. Arguing as before, we deduce that $\dd$ is a linear combination of $\dd_i$ and $\dd_1$, and $\dd$ is also a linear combination of $\dd_i$ and $\dd_2$ for all $3 \leq i \leq t$. So, $\dd_i$ is a linear combination of  $\dd_1$ and $\dd_2$ for all $i$, and this implies that $\mathbf H(\dd_1) \cap \mathbf H(\dd_2)=\bigcap_{i=1}^t \mathbf H(\dd_i)$. The proof of the claim now follows again from Lemma \ref{face-Eff-lemma}.
\end{proof}

There are three possible cases that we need to distinguish:

\underline{Case 1:} $q_{\Lambda}(\dd_1)=q_{\Lambda}(\dd_2)=0$. First note that $\langle \dd_1,\dd_2 \rangle_{\Lambda}+\langle \dd_2,\dd_1 \rangle_{\Lambda}=0$ since $\dd_1$ is in the radical of $\chi_{\Lambda}$ by Lemma \ref{iso-RR-lemma}.

If $\rk_{\Lambda} (\dd_1) \cdot \rk_{\Lambda}(\dd_2)\geq 0$ then $\langle \dd_i,\dd_j\rangle_{\Lambda}=-\dim_k \Ext^1_{\Lambda}(E_i,E_j), \forall 1 \leq i\neq j\leq 2$.  Consequently, $\langle \dd_1,\dd_2\rangle_{\Lambda}=\langle \dd_2,\dd_1 \rangle_{\Lambda}=0$. But then the two isotropic roots $\dd_1$ and $\dd_2$ would have to be multiple of each other. So, $\dd_1=\dd_2$ which is a contradiction.

If $\rk_{\Lambda}(\dd_1)>0>\rk_{\Lambda}(\dd_2)$ or $\rk_{\Lambda}(\dd_2)>0>\rk_{\Lambda}(\dd_1)$ then either $\langle \dd_1,\dd_2\rangle_{\Lambda}=0$ or $\langle \dd_2,\dd_1\rangle_{\Lambda}=0$. Since we are in the isotropic case, this would again imply that $\dd_1=\dd_2$. But this is a contradiction.

\underline{Case 2:} $q_{\Lambda}(\dd_1)=1, q_{\Lambda}(\dd_2)=0$ (or the other way around). Using the claim above and the fact that $q_{\Lambda}(\dd)=0$, we deduce that $n_1^2=n_1n_2(-\langle \dd_1,\dd_2\rangle_{\Lambda}-\langle \dd_2,\dd_1\rangle_{\Lambda})$. This relation combined with the fact that $\dd_2$ is in the radical of $\chi_{\Lambda}$ implies that $n_1=0$, which is a contradiction.

\underline{Case 3:} $q_{\Lambda}(\dd_1)=1, q_{\Lambda}(\dd_2)=1$. In this case, both $E_1$ and $E_2$ are exceptional $\Lambda$-modules. To simplify notation, set $l=-\langle \dd_1,\dd_2\rangle_{\Lambda}-\langle \dd_2,\dd_1\rangle_{\Lambda}$. Then, using the fact that $\dd$ is an isotropic root in the radical of $\chi_{\Lambda}$, we deduce that $2n_1=n_2l, 2n_2=n_1l$, and $n_1^2+n_2^2=ln_1n_2$ . It is now easy to see that $n_1=n_2=1$ and $l=2$. Without loss of generality, we can assume that $E_1$ is a submodule of $M$ and $E_2=M/E_1$. Then, we have that $\dim_k \Ext^1_{\Lambda}(E_2,E_1)>0$.

We also note that $\langle \dd_1,\dd_2 \rangle_{\Lambda} \neq -1$ and $\langle \dd_2,\dd_1 \rangle_{\Lambda} \neq -1$ since otherwise both of these two inner products would have to be $-1$, and this would imply that $\theta_{\dd}(\ddim E_1)=0$. But this would contradict the fact that $M$ is $\theta_{\dd}$-stable.

\hspace{15pt}\underline{Case 3.1:} $\rk_{\Lambda}(\dd_1)\cdot \rk_{\Lambda}(\dd_2)>0$.  In this case, we have that $\langle \dd_1,\dd_2 \rangle_{\Lambda}=-\dim_k \Ext^1_{\Lambda}(E_1,E_2)$ and $\langle \dd_2,\dd_1 \rangle_{\Lambda}=-\dim_k \Ext^1_{\Lambda}(E_2,E_1)$. It now follows that $(E_1,E_2)$ is an orthogonal exceptional sequence with the desired properties.

\hspace{15pt}\underline{Case 3.2:} $\rk_{\Lambda}(\dd_1)\cdot \rk_{\Lambda}(\dd_2)\leq 0$. First, note that $\rk_{\Lambda}(\dd_1)$  and $\rk_{\Lambda}(\dd_2)$ can not be both zero since otherwise $\rk_{\Lambda}(\dd)$ would be zero, and this would imply that $\dd=\hh$. But, then $\theta_{\dd}(\ddim E_1)=0$ which is not possible.

It is now easy to see that properties (ii)-(iii) mentioned at the beginning of this section imply that $(E_1,E_2)$ is an orthogonal exceptional sequence with the desired properties.
\end{proof}

Now, we are ready to prove Theorem \ref{reptype-canonical-thm}.

\begin{proof}[Proof of Theorem \ref{reptype-canonical-thm}] The implication $(2) \Longrightarrow (1)$ has been proved in Proposition \ref{qt-ssc-prop}.

Now, let us assume that $\Lambda$ is a tame canonical algebra and let $\dd$ be a generic root of $\Lambda$. We know from Theorem \ref{BS-GS-irr-comp-thm} that $\dd$ is a Schur root and  $\rep(\Lambda,\dd)$ has a unique indecomposable irreducible component, call it $C$.

If $q_{\Lambda}(\dd)=1$ then $k(C)^{\GL(\dd)}\simeq k$ since $C$ is an orbit closure in this case.

It remains to look into the case when $\dd$ is an isotropic Schur root of $\Lambda$. It follows from Proposition \ref{exceptional-seq-tubular-prop} that there exists an orthogonal exceptional sequence $\mathcal E=(E_1,E_2)$ such that $C\cap  \filt_{\dd}(\mathcal E) \neq \emptyset$ and $\Lambda_{\mathcal E}$ is the path algebra of $K_2$. Now, applying the reduction Theorem \ref{rational-inv-quiverel-thm}, we conclude that $k(C)^{\GL(\dd)}\simeq k(\rep(K_2,(1,1)))^{\GL((1,1))}\simeq k(t)$.
\end{proof}

Finally, let us prove Proposition \ref{ratio-inv-tame-can-prop}:

\begin{proof}[Proof of Proposition \ref{ratio-inv-tame-can-prop}] We know from Theorem \ref{reptype-quivers-thm} and Theorem \ref{reptype-canonical-thm} that if $C$ is an indecomposable irreducible component of $\rep(\Lambda,\dd)$ then $S^m(k(C)^{\GL(\dd)})$ is isomorphic to either $k$ in case $\dd$ is a real Schur root or to $k(t_1,\ldots, t_m)$ in case $\dd$ is isotropic. The proof now follows from Proposition \ref{rational-inv-generic-decomp-prop}.
\end{proof}

\begin{remark} In a sequel to the current work, we plan to use a similar strategy to prove the analogous of Theorem \ref{reptype-canonical-thm} for other classes of algebras, including the class of quasi-tilted algebras and of string algebras. Of course, the ultimate goal is to prove Theorem \ref{reptype-canonical-thm} for arbitrary tame algebras. Since it is believed that the representation theory of tame algebras can be reduced to that of tame strongly simply-connected algebras via covering functors, the next natural steps are: (1) to solve the rationality problem for tame strongly simply-connected algebras; (2) to show that the rationality of fields of rational invariants is preserved under covering functors (in the relevant cases). We plan to address all these problems in future work in which the reduction Theorem \ref{rational-inv-quiverel-thm} combined with the systematic approach to finding short orthogonal exceptional sequences from Section \ref{exceptional-effective-wts-sec} will play a fundamental role. 
\end{remark}

\section*{Acknowledgment} I would like to thank Frauke Bleher, Grzegorz Bobi{\'n}ski, Harm Derksen, Christof Geiss, Dirk Kusin, Hagen Meltzer, Steven Sam, and Jerzy Weyman for helpful conversations on the subject of the paper.

%\bibliography{biblio}\label{biblio-sec}

\begin{thebibliography}{10}\label{biblio-sec}

\bibitem{ALeB}
J.~Adriaenssens and L.~Le~Bruyn.
\newblock Local quivers and stable representations.
\newblock {\em Comm. Algebra}, 31(4):1777--1797, 2003.

\bibitem{ArHa2}
I.V. Arzhantsev and J.~Hausen.
\newblock Geometric invariant theory via cox rings.
\newblock Preprint, ar\textrm{X}iv:0706.4353v1, 2007.

\bibitem{Ati}
M.~F. Atiyah.
\newblock Vector bundles over an elliptic curve.
\newblock {\em Proc. London Math. Soc. (3)}, 7:414--452, 1957.

\bibitem{Benson}
D.~J. Benson.
\newblock {\em Representations and cohomology. {I}: {B}asic representation
  theory of finite groups and associative algebras}, volume~30 of {\em
  Cambridge Studies in Advanced Mathematics}.
\newblock Cambridge University Press, Cambridge, second edition, 1998.

\bibitem{Bob2}
G.~Bobi{\'n}ski.
\newblock Geometry and the zero sets of semi-invariants for homogeneous modules
  over canonical algebras.
\newblock {\em J. Algebra}, 319(3):1320--1335, 2008.

\bibitem{Bob1}
G.~Bobi{\'n}ski.
\newblock On the zero set of semi-invariants for regular modules over tame
  canonical algebras.
\newblock {\em J. Pure Appl. Algebra}, 212(6):1457--1471, 2008.

\bibitem{BobRieSko}
G.~Bobi{\'n}ski, Ch. Riedtmann, and A.~Skowro{\'n}ski.
\newblock Semi-invariants of quivers and their zero sets.
\newblock In {\em Trends in representation theory of algebras and related
  topics}, EMS Ser. Congr. Rep., pages 49--99. Eur. Math. Soc., Z\"urich, 2008.

\bibitem{BS1}
G.~Bobi{\'n}ski and A.~Skowro{\'n}ski.
\newblock Geometry of modules over tame quasi-tilted algebras.
\newblock {\em Colloq. Math.}, 79(1):85--118, 1999.

\bibitem{BobSko}
G.~Bobi{\'n}ski and A.~Skowro{\'n}ski.
\newblock Geometry of periodic modules over tame concealed and tubular
  algebras.
\newblock {\em Algebr. Represent. Theory}, 5(2):187--200, 2002.

\bibitem{Bock}
R.~Bocklandt.
\newblock Smooth quiver representation spaces.
\newblock {\em J. Algebra}, 253(2):296--313, 2002.

\bibitem{Bon}
K.~Bongartz.
\newblock Algebras and quadratic forms.
\newblock {\em J. London Math. Soc. (2)}, 28(3):461--469, 1983.

\bibitem{BruPenSko}
Th. Br{\"u}stle, J.~A. de~la Pe{\~n}a, and A.~Skowronski.
\newblock Tame algebras and {T}its quadratic forms.
\newblock To appear.

\bibitem{CC5}
C.~Chindris.
\newblock On \textrm{GIT}-fans for quivers.
\newblock Preprint avilable at ar\textrm{X}iv:0805.1440v1 [math.RT].

\bibitem{CDW}
C.~Chindris, H.~Derksen, and J.~Weyman.
\newblock Counterexamples to {O}kounkov's log-concavity conjecture.
\newblock {\em Compositio Mathematica}, 143(6):1545--1557, 2007.

\bibitem{C-BS}
W.~Crawley-Boevey and J.~Schr{\"o}er.
\newblock Irreducible components of varieties of modules.
\newblock {\em J. Reine Angew. Math.}, 553:201--220, 2002.

\bibitem{delaP}
J.~A. de~la Pe{\~n}a.
\newblock On the dimension of the module-varieties of tame and wild algebras.
\newblock {\em Comm. Algebra}, 19(6):1795--1807, 1991.

\bibitem{DW1}
H.~Derksen and J.~Weyman.
\newblock Semi-invariants of quivers and saturation for
  \textrm{L}ittlewood-\textrm{R}ichardson coefficients.
\newblock {\em J. Amer. Math. Soc.}, 13(3):467--479, 2000.

\bibitem{DW3}
H.~Derksen and J.~Weyman.
\newblock On the {L}ittlewood-{R}ichardson polynomials.
\newblock {\em J. Algebra}, 255(2):247--257, 2002.

\bibitem{DW2}
H.~Derksen and J.~Weyman.
\newblock The combinatorics of quiver representations.
\newblock Preprint available at arXiv.math.RT/0608288, 2006.

\bibitem{Domo2}
M.~Domokos.
\newblock On singularities of quiver moduli.
\newblock Preprint, arXiv:0903.4139v2 [math.RT], 2009.

\bibitem{DL}
M.~Domokos and H.~Lenzing.
\newblock Invariant theory of canonical algebras.
\newblock {\em J. Algebra}, 228(2):738--762, 2000.

\bibitem{DL2}
M.~Domokos and H.~Lenzing.
\newblock Moduli spaces for representations of concealed-canonical algebras.
\newblock {\em J. Algebra}, 251(1):371--394, 2002.

\bibitem{DF}
P.~Donovan and M.~R. Freislich.
\newblock {\em The representation theory of finite graphs and associated
  algebras}.
\newblock Number~5 in Carleton Mathematical Lecture Notes. Carleton University,
  Ottawa, Ont., 1973.

\bibitem{Dro}
Yu.~A. Drozd.
\newblock Tame and wild matrix problems.
\newblock In {\em Representations and quadratic forms (Russian)}, pages 39--74,
  154. Akad. Nauk Ukrain. SSR Inst. Mat., Kiev, 1979.

\bibitem{Eisbook}
D.~Eisenbud.
\newblock {\em Commutative algebra}, volume 150 of {\em Graduate Texts in
  Mathematics}.
\newblock Springer-Verlag, New York, 1995.
\newblock With a view toward algebraic geometry.

\bibitem{For}
E.~Formanek.
\newblock The ring of generic matrices.
\newblock {\em J. Algebra}, 258(1):310--320, 2002.
\newblock Special issue in celebration of Claudio Procesi's 60th birthday.

\bibitem{Ga}
P.~Gabriel.
\newblock Unzerlegbare \textrm{D}arstellungen. \textrm{I}. (\textrm{G}erman.
  \textrm{E}nglish summary).
\newblock {\em Manuscripta Math.}, 6(2):71--103 ; correction, ibid. 6 (1972),
  309., 1972.

\bibitem{GeiLen}
W.~Geigle and H.~Lenzing.
\newblock A class of weighted projective curves arising in representation
  theory of finite-dimensional algebras.
\newblock In {\em Singularities, representation of algebras, and vector bundles
  ({L}ambrecht, 1985)}, volume 1273 of {\em Lecture Notes in Math.}, pages
  265--297. Springer, Berlin, 1987.

\bibitem{GeiSch}
Ch. Geiss and J.~Schr{\"o}er.
\newblock Varieties of modules over tubular algebras.
\newblock {\em Colloq. Math.}, 95(2):163--183, 2003.

\bibitem{HapVos}
D.~Happel and D.~Vossieck.
\newblock Minimal algebras of infinite representation type with preprojective
  component.
\newblock {\em Manuscripta Math.}, 42(2-3):221--243, 1983.

\bibitem{Har}
R.~Hartshorne.
\newblock {\em Algebraic geometry}.
\newblock Springer-Verlag, New York, 1977.
\newblock Graduate Texts in Mathematics, No. 52.

\bibitem{IOTW}
K.~Igusa, K.~Orr, G.~Todorov, and J.~Weyman.
\newblock Cluster complexes via semi-invariants.
\newblock {\em Compos. Math.}, 145(4):1001--1034, 2009.

\bibitem{Kac2}
V.G. Kac.
\newblock Infinite root systems, representations of graphs and invariant
  theory.
\newblock {\em Invent. Math}, 56(1):57--92, 1980.

\bibitem{Kac}
V.G. Kac.
\newblock Infinite root systems, representations of graphs and invariant theory
  \textrm{II}.
\newblock {\em J. Algebra}, 78(1):141--162, 1982.

\bibitem{Kel2}
B.~Keller.
\newblock {$A$}-infinity algebras in representation theory.
\newblock In {\em Representations of algebra. Vol. I, II}, pages 74--86.
  Beijing Norm. Univ. Press, Beijing, 2002.

\bibitem{Kel1}
B.~Keller.
\newblock {$A$}-infinity algebras, modules and functor categories.
\newblock In {\em Trends in representation theory of algebras and related
  topics}, volume 406 of {\em Contemp. Math.}, pages 67--93. Amer. Math. Soc.,
  Providence, RI, 2006.

\bibitem{K}
A.D. King.
\newblock Moduli of representations of finite-dimensional algebras.
\newblock {\em Quart. J. Math. Oxford Ser.(2)}, 45(180):515--530, 1994.

\bibitem{KraRie}
H.~Kraft and Ch. Riedtmann.
\newblock Geometry of representations of quivers.
\newblock In {\em Representations of algebras ({D}urham, 1985)}, volume 116 of
  {\em London Math. Soc. Lecture Note Ser.}, pages 109--145. Cambridge Univ.
  Press, Cambridge, 1986.

\bibitem{Kra}
W.~Kra\'skiewicz.
\newblock On semi-invariants of tilted algebras of type
  \textrm{$\mathbb{A}_n$}.
\newblock {\em Colloq. Math.}, 90(2):253--267, 2001.

\bibitem{LeB}
L.~Le~Bruyn.
\newblock Centers of generic division algebras, the rationality problem
  1965--1990.
\newblock {\em Israel J. Math.}, 76(1-2):97--111, 1991.

\bibitem{Len}
H.~Lenzing.
\newblock A {$K$}-theoretic study of canonical algebras.
\newblock In {\em Representation theory of algebras ({C}ocoyoc, 1994)},
  volume~18 of {\em CMS Conf. Proc.}, pages 433--454. Amer. Math. Soc.,
  Providence, RI, 1996.

\bibitem{LenMel}
H.~Lenzing and H.~Meltzer.
\newblock Sheaves on a weighted projective line of genus one, and
  representations of a tubular algebra [ {MR}1206953 (94d:16019)].
\newblock In {\em Representations of algebras ({O}ttawa, {ON}, 1992)},
  volume~14 of {\em CMS Conf. Proc.}, pages 313--337. Amer. Math. Soc.,
  Providence, RI, 1993.

\bibitem{Naz}
L.A. Nazarova.
\newblock Representations of quivers of infinite type. (\textrm{R}ussian).
\newblock {\em Izv. Akad. Nauk SSSR Ser. Mat.}, 37:752--791, 1973.

\bibitem{Rei1}
Z.~Reichstein.
\newblock On the notion of essential dimension for algebraic groups.
\newblock {\em Transform. Groups}, 5(3):265--304, 2000.

\bibitem{ReiVon1}
Z.~Reichstein and N.~Vonessen.
\newblock Tame group actions on central simple algebras.
\newblock {\em J. Algebra}, 318(2):1039--1056, 2007.

\bibitem{RieSch1}
Ch. Riedtmann and A.~Schofield.
\newblock On open orbits and their complements.
\newblock {\em J. Algebra}, 130(2):388--411, 1990.

\bibitem{R3}
C.~Ringel.
\newblock {\em Tame algebras and integral quadratic forms}, volume 1099 of {\em
  Lecture Notes in Mathematics}.
\newblock Springer-Verlag, Berlin, 1984.

\bibitem{R4}
C.~M. Ringel.
\newblock The rational invariants of the tame quivers.
\newblock {\em Invent. Math.}, 58(3):217--239, 1980.

\bibitem{R}
C.M. Ringel.
\newblock Representations of $\textit{K}$-species and bimodules.
\newblock {\em J. Algebra}, 41(2):269--302, 1976.

\bibitem{R2}
C.M. Ringel.
\newblock The braid group action on the set of exceptional sequences of a
  hereditary {A}rtin algebra.
\newblock In {\em Abelian group theory and related topics (Oberwolfach, 1993)},
  volume 171 of {\em Contemp. Math.}, pages 339--352. Amer. Math. Soc.,
  Providence, RI, 1994.

\bibitem{Ros}
M.~Rosenlicht.
\newblock Some basic theorems on algebraic groups.
\newblock {\em Amer. J. Math.}, 78:401--443, 1956.

\bibitem{S1}
A.~Schofield.
\newblock General representations of quivers.
\newblock {\em Proc. London Math. Soc. (3)}, 65(1):46--64, 1992.

\bibitem{S3}
A.~Schofield.
\newblock Birational classification of moduli spaces of representations of
  quivers.
\newblock {\em Indag. Math. (N.S.)}, 12(3):407--432, 2001.

\bibitem{PopVin}
I.~R. Shafarevich, editor.
\newblock {\em Algebraic geometry. {IV}}, volume~55 of {\em Encyclopaedia of
  Mathematical Sciences}.
\newblock Springer-Verlag, Berlin, 1994.
\newblock Linear algebraic groups. Invariant theory, A translation of {\it
  Algebraic geometry. 4} (Russian), Akad. Nauk SSSR Vsesoyuz. Inst. Nauchn. i
  Tekhn. Inform., Moscow, 1989 [ MR1100483 (91k:14001)], Translation edited by
  A. N. Parshin and I. R. Shafarevich.

\bibitem{Sha}
I.~R. Shafarevich.
\newblock {\em Basic algebraic geometry. 1}.
\newblock Springer-Verlag, Berlin, second edition, 1994.
\newblock Varieties in projective space, Translated from the 1988 Russian
  edition and with notes by Miles Reid.

\bibitem{SW2}
A.~Skowro{\'n}ski and J.~Weyman.
\newblock Semi-invariants of canonical algebras.
\newblock {\em Manuscripta Math.}, 100(3):391--403, 1999.

\bibitem{SW1}
A.~Skowro{\'n}ski and J.~Weyman.
\newblock The algebras of semi-invariants of quivers.
\newblock {\em Transform. Groups}, 5(4):361--402, 2000.

\bibitem{Wat}
C~E. Watts.
\newblock Intrinsic characterizations of some additive functors.
\newblock {\em Proc. Amer. Math. Soc.}, 11:5--8, 1960.

\bibitem{Jerzy}
J.~Weyman.
\newblock Personal communication.

\end{thebibliography}

\end{document}